\documentclass[11pt,
]{amsart}
\usepackage{%amsmath,
amssymb, graphicx
}
\usepackage{mathrsfs} 
\usepackage{amsthm}
\usepackage{graphicx,color}

\newtheorem{theorem}{Theorem}
\newtheorem{lemma}{Lemma}
\newtheorem{proposition}{Proposition}

\theoremstyle{remark}

\date{\today}

\title[A nonlinear elastic wave equation]{Determination of the density in a nonlinear elastic wave equation}
  \author[G. Uhlmann]{Gunther Uhlmann}
\address{Department of Mathematics, University of Washington, Seattle, WA 98195, USA; Institute for Advanced Study, 
The Hong Kong University of Science and Technology, Kowloon, Hong Kong, China (\tt{gunther@math.washington.edu})
}
  
  \author[J. Zhai]{Jian Zhai}
\address{School of Mathematical Sciences,
  Fudan University, Shanghai 200433, China
  (\tt{jianzhai@fudan.edu.cn}).}
    \thanks{J. Zhai is supported by National Key Research and Development Programs of China (No. 2023YFA1009103), Science and Technology Commission of Shanghai Municipality (23JC1400501)}

\begin{document}
\begin{abstract}
This is a continuation of our study \cite{uhlmann2021inverse} on an inverse boundary value problem for a nonlinear elastic wave equation. We prove that all the linear and nonlinear coefficients can be recovered from the displacement-to-traction map, including the density, under some natural geometric conditions on the wavespeeds.
\end{abstract}
\keywords{elastic waves, inverse boundary value problem, quasilinear equation, Gaussian beams}
\maketitle

\section{Introduction}
Let $\Omega\subset\mathbb{R}^3$ be a bounded domain with smooth boundary $\partial\Omega$. Denote $x=(x_1,x_2,x_3)$ to be the Cartesian coordinates. Here $\Omega$ represents an elastic, nonhomogeneous, isotropic object with Lam\'e parameters $\lambda,\mu$ and density $\rho$. We take into account the nonlinear behavior of elastic medium in this article. Let the vector $u=(u_1,u_2,u_3)$ be the displacement vector, and the (nonlinear) strain tensor is
\[
\varepsilon_{ij}(u)=\frac{1}{2}\left(\frac{\partial u_i}{\partial x_j}+\frac{\partial u_j}{\partial x_i}+\frac{\partial u_k}{\partial x_i}\frac{\partial u_k}{\partial x_j}\right).
\]
We consider the nonlinear model of elasticity as in \cite{de2018nonlinear}, whereas the stress tensor is of the form,
%\begin{equation}
%\begin{split}
%S_{ij}=&\lambda \varepsilon_{mm}\left(\delta_{ij}+\frac{\partial u_i}{\partial x_j}\right)+2\mu\left(\varepsilon_{ij}+\varepsilon_{jn}\frac{\partial u_i}{\partial x_n}\right)\\
%&+ \mathscr{A}\varepsilon_{in}\varepsilon_{jn}+\mathscr{B}(2\varepsilon_{nn}\varepsilon_{ij}+\varepsilon_{mn}\varepsilon_{mn}\delta_{ij})+\mathscr{C}\varepsilon_{mm}\varepsilon_{nn}\delta_{ij}.
%\end{split}
%\end{equation}
\begin{equation}\label{Sform}
\begin{split}
S_{ij}=&\lambda \varepsilon_{mm}\delta_{ij}+\lambda\widetilde{\varepsilon}_{mm}\frac{\partial u_i}{\partial x_j}+2\mu\left(\varepsilon_{ij}+\widetilde{\varepsilon}_{jn}\frac{\partial u_i}{\partial x_n}\right)\\
&+ \mathscr{A}\widetilde{\varepsilon}_{in}\widetilde{\varepsilon}_{jn}+\mathscr{B}(2\widetilde{\varepsilon}_{nn}\widetilde{\varepsilon}_{ij}+\widetilde{\varepsilon}_{mn}\widetilde{\varepsilon}_{mn}\delta_{ij})+\mathscr{C}\widetilde{\varepsilon}_{mm}\widetilde{\varepsilon}_{nn}\delta_{ij}+\mathcal{O}(u^3),
\end{split}
\end{equation}
where $\widetilde{\varepsilon}$ is the linearized strain tensor
\[
\widetilde{\varepsilon}_{ij}(u)=\frac{1}{2}\left(\frac{\partial u_i}{\partial x_j}+\frac{\partial u_j}{\partial x_i}\right),
\]
For the rest of the paper, we will just discard the $\mathcal{O}(u^3)$ terms in \eqref{Sform}.\\

We consider the initial boundary value problem
\begin{equation}\label{elastic_eq}
\begin{split}
&\rho\frac{\partial^2u}{\partial t^2}-\nabla\cdot S(x,u)=0,\quad (t,x)\in (0,T)\times\Omega,\\
&u(t,x)=f(t,x),\quad (t,x)\in (0,T)\times\partial \Omega,\\
&u(0,x)=\frac{\partial}{\partial t}u(0,x)=0,\quad x\in \Omega,
\end{split}
\end{equation}
where $S$ is of the form \eqref{Sform}.
Denote the displacement-to-traction map as
 \[
\Lambda:f=u\vert_{(0,T)\times\partial\Omega}\rightarrow\nu\cdot S(x,u) \vert_{(0,T)\times\partial\Omega},
\]
where $\nu$ is the exterior normal unit vector to $\partial \Omega$. We consider the inverse problem of determining $\lambda,\mu,\rho,\mathscr{A},\mathscr{B},\mathscr{C}$ from $\Lambda$.  The well-posedness of the Dirichlet problem with small boundary data is established in \cite{de2018nonlinear}. \\

Closely related is the inverse boundary value problem for the linear elastic wave equations
\begin{equation}\label{elastic_linear_eq}
\begin{split}
&\rho\frac{\partial^2u}{\partial t^2}-\nabla\cdot S^L(x,u)=0,\quad (t,x)\in (0,T)\times\Omega,\\
&u(t,x)=f(t,x),\quad (t,x)\in (0,T)\times\partial \Omega,\\
&u(0,x)=\frac{\partial}{\partial t}u(0,x)=0,\quad x\in \Omega,
\end{split}
\end{equation}
for which one wants to recover $\lambda,\mu,\rho$
from the (linearized) Dirichlet-to-Neumann (DtN) map defined as
\[
\Lambda^{lin}:f= u\vert_{(0,T)\times\partial\Omega}\rightarrow \nu\cdot S^L(x,u)\vert_{(0,T)\times\partial\Omega}.
\]
Here $S^L$ is the linearized stress
\[
S^L_{ij}(x,u)=\lambda \widetilde{\varepsilon}_{mm}\delta_{ij}+2\mu \widetilde{\varepsilon}_{ij}.
\]

Throughout the paper, we assume the strong ellipticity condition
\begin{equation}\label{elliptiticity}
\mu>0,\quad 3\lambda+2\mu>0~\text{on}~\overline{\Omega}.
\end{equation}
Denote
\[
c_P=\sqrt{\frac{\lambda+2\mu}{\rho}},\quad c_S=\sqrt{\frac{\mu}{\rho}},
\]
which are actually the wavespeeds for \textit{P}- and \textit{S}- waves respectively. By the assumption \eqref{elliptiticity}, we have $c_P>c_S$ in $\overline{\Omega}$. Denote the Riemannian metrics associated with the $P/S$- wavespeeds to be
\[
g_{P/S}=c_{P/S}^{-2}\mathrm{d}s^2,
\]
where $\mathrm{d}s^2$ is the Euclidean metric. Now, one can consider $(\Omega,g_\bullet)$ as a compact Riemannian manifold with boundary $\partial\Omega$, where $\bullet=P,S$.

We assume that $(\Omega,g_\bullet)$ is non-trapping and $\partial\Omega$ is convex with respect to the metric $g_\bullet$.
Denote $\mathrm{diam}_{P/S}(\Omega)$ be the diameter of $\Omega$ with respect to $g_{P/S}$. More precisely, 
\[
\mathrm{diam}_{P/S}(\Omega)=\sup\{\textit{lengths of all geodesics in }(\Omega,g_{P/S})\}.
\]

The main result of this paper is the following theorem.
\begin{theorem}
Assume $\partial\Omega$ is strictly convex with respect to $g_{P/S}$, and either of the following conditions holds
\begin{enumerate}
\item $(\Omega,g_{P/S})$ is simple;
\item $(\Omega,g_{P/S})$ satisfies the foliation condition.
\end{enumerate}
If $T>2\,\max\{\mathrm{diam}_S(\Omega),\mathrm{diam}_P(\Omega)\}$, then $\Lambda$ determines $\lambda,\mu,\rho,\mathscr{A},\mathscr{B},\mathscr{C}$ in $\overline{\Omega}$ uniquely.
\end{theorem}

Assume $(\Omega,g)$ is a compact Riemannian manifold with strictly convex boundary $\partial \Omega$.
We recall that (1) $(\Omega,g)$ called to be simple if any two points $x,y\in\Omega$ can be connected by a unique geodesic, contained in $\Omega$, depending smoothly on $x$ and $y$. (2) $(\Omega,g)$ satisfies the foliation condition if there is a smooth strictly convex function $f:M\rightarrow \mathbb{R}$. For more discussions on the foliation condition, we refer to \cite{paternain2019geodesic}. We emphasize here that either of the two geometrical conditions implies that $(\Omega,g)$ is non-trapping. The foliation condition is satisfies if the manifold has negative sectional curvatures, or non-positive sectional curvatures if the manifolds is simply connected. The foliation condition allows for conjugate points.

For the inverse boundary value problem for the \textit{linear} elastic wave equation \eqref{elastic_linear_eq}, the uniqueness of the wavespeeds $c_{P/S}$, under the above two geometrical conditions, is proved in \cite{rachele2000inverse,stefanov2017local} respectively. The problem is reduced to the lens rigidity problem for a Riemannian metric which is conformally Euclidean. The related lens rigidity problem is solved under the simplicity condition in \cite{muhometov1978problem} and under the foliation condition in \cite{stefanov2016boundary}. To recover $\lambda,\mu,\rho$ simultaneously, the uniqueness of $\rho$ needs to be proved additionally. For the uniqueness of the density, the problem can be reduced to the geodesic ray transform of a 2-tensor using \textit{P}-wave measurements (cf. \cite{rachele2003uniqueness,bhattacharyya2018local}), but this reduction requires a technical assumption $c_P\neq 2c_S$. The $s$-injectivity of this geodesic ray transform is proved under the strictly convex foliation condition \cite{stefanov2018inverting}. Under simplicity condition, the $s$-injectivity is proved under extra curvature conditions \cite{sharafutdinov1992integral,dairbekov2006integral,paternain2015invariant}. 

In our previous work \cite{uhlmann2021inverse}, we proved the uniqueness of $\mathscr{A}, \mathscr{B},\mathscr{C}$ assuming $\lambda,\mu,\rho$ are known. For the uniqueness of $\rho$, extra conditions are needed if one resorts to the result for the linear equation as mentioned above. In this paper, we will utilize the nonlinearity to show the uniqueness of $\rho$, without any additional assumptions.

The proof is based on the construction of Gaussian beam solutions. Compared to our previous work \cite{uhlmann2021inverse}, since now the density is not known, the reflection of waves at the boundary is hard to control. We remark here that elastic waves would undergo a $P/S$ mode conversion when reflected. We refer to the work \cite{stefanov2021transmission} for this reflection behavior. The reflection of Gaussian beam solutions is carefully characterized below in Section \ref{gaussianreflection}.\\

The nonlinear interaction of distorted planes was first used in \cite{kurylev2018inverse} to recover a Lorentzian metric in a semilinear wave equation. Since then, many inverse problems for nonlinear hyperbolic equations have been studied. See \cite{uhlmann2021inverse1} for a review. We also refer to \cite{lassas2018inverse,de2018nonlinear,uhlmann2018determination,feizmohammadi2019recovery,uhlmann2021inverse,kurylev2014inverse,wang2019inverse,chen2021inverse,chen2021detection,sa2021singularities,lassas2021stability,lassas2020uniqueness,balehowsky2022inverse,barreto2021recovery,sa2022recovery,hintz2022inverse,hintz2022dirichlet,uhlmann2022inverse,chen2022retrieving} and the references therein. \\

The rest of the paper is organized as follows. In Section \ref{gaussianbeams}, we construct Gaussian beam solutions for the linear elastic wave equations. In Section \ref{gaussianreflection}, the reflection of Gaussian beams on the boundary is carefully characterized. In Section \ref{mainproof}, we give the proof of the main result.
\section{Gaussian beams}\label{gaussianbeams}
In this section, we construct Gaussian beam solutions to the \textit{linear} elastic wave equation
\[
\mathcal{L}_{\lambda,\mu,\rho}u:=\rho\frac{\partial^2u}{\partial t^2}-\nabla\cdot S^L(x,u)=0,
\]
Much of the work has been carried out in \cite{uhlmann2021inverse}.
Gaussian beams have also been used to study various inverse problems for both elliptic and hyperbolic equations \cite{katchalov1998multidimensional,bao2014sensitivity, belishev1996boundary, dos2016calderon,feizmohammadi2019timedependent,feizmohammadi2019recovery,feizmohammadi2019inverse}. For the construction, we need to introduce the Fermi coordinates.
\subsection{Fermi coordinates}
Denote
\[
M=\mathbb{R}\times\Omega,
\]
and consider $M$ as a Lorentzian manifold with metric $\overline{g}_\bullet:=-\mathrm{d}t^2+g_\bullet=-\mathrm{d}t^2+c_\bullet^{-2}\mathrm{d}s^2$, where $\bullet=P$ or $S$. We can extend $g_\bullet$ smoothly to a slightly larger domain $\Omega'$, such that $\Omega\subset\subset\Omega'$, and consider the Lorentzian manifold $(M',\overline{g}_\bullet):=(\mathbb{R}\times \Omega',\overline{g}_\bullet)$. 

Consider a null-geodesic $\vartheta$ in $(M',\overline{g}_\bullet)$. Notice that $\vartheta$ can be expressed as $\vartheta(t)=(t,\gamma(t))$, where $\gamma$ is a unit-speed geodesic in the Riemannian manifold $(\Omega',g_\bullet)$. Assume that $\vartheta$ passes through two boundary points $(t_-,\gamma(t_-))$ and $(t_+,\gamma(t_+))$ of $\partial M$, where $t_-,t_+\in (0,T)$, $t_-<t_+$ and $\gamma(t_-),\gamma(t_+)\in\partial\Omega$. We assume $\vartheta:[t_--\epsilon,t_++\epsilon]\rightarrow M'$, by extending it if necessary, and introduce the Fermi coordinates in a neighborhood of $\vartheta$. We will follow the construction of the coordinates in \cite{feizmohammadi2019timedependent}. See also \cite{kurylev2014inverse}, \cite{uhlmann2018determination}.

Assume $\gamma(t_0)=x_0\in \Omega$ where $t_0\in(t_-,t_+)$. Choose $\alpha_2,\alpha_3\in T_{x_0}\Omega'$ such that $\{\dot{\gamma}(t_0),\alpha_2,\alpha_3\}$ forms an orthonormal basis for $T_{x_0}\Omega'$. Let $s$ denote the arc length along $\gamma$ from $x_0$. We note here that $s$ can be positive or negative, and $(t_0+s,\gamma(t_0+s))=\vartheta(t_0+s)$. For $k=2,3$, let $e_k(s)\in T_{\gamma(t_0+s)}\Omega'$ be the parallel transport of the vector $\alpha_k$ along $\gamma$ to the point $\gamma(t_0+s)$.

Define the coordinate system $(y^0=t,y^1=s,y^2,y^3)$ through $\mathcal{F}_1:\mathbb{R}^{1+3}\rightarrow \mathbb{R}\times\Omega'$:
\[
\mathcal{F}_1(y^0=t,y^1=s,y^2,y^3)=(t,\exp_{\gamma(t_0+s)}\left(y^2e_2(s)+y^3e_3(s)\right)).
\]
Then the null-geodesic $\vartheta$ can be expressed as 
\[
\vartheta=\{t-s=t_0,y^2=y^3=0\}.
\]
Notice that $(y^1=s,y^2,y^3)$ gives a coordinate system on $\Omega'$ in a neighborhood of $\gamma$ such that
\[
g_\bullet\vert_{\gamma}=\sum_{j=1}^3(\mathrm{d}y^j)^2,\quad\text{and}\quad\frac{\partial g_{\bullet,jk}}{\partial y^i}\Big\vert_\gamma=0,~1\leq i,j,k\leq 3.
\]
Under this coordinate system, the Euclidean metric $g_E:=\mathrm{d}s^2$ in a neighborhood of $\gamma$ takes the form
\[
g_E=c_\bullet^2g_\bullet=\sum_{1\leq i,j\leq 3}c^2_\bullet g_{\bullet,ij}\mathrm{d}y^i\mathrm{d}y^j,
\]
and the Christoffel symbols for the Euclidean metric are given by
\begin{equation}\label{Christoffel}
\begin{split}
&\Gamma_{\alpha\beta}^1=-c_\bullet^{-1}\frac{\partial c_\bullet}{\partial s} g_{\bullet,\alpha\beta},\quad \Gamma_{1\alpha}^\beta=\delta^\alpha_\beta c_\bullet^{-1}\frac{\partial c_\bullet}{\partial s},\quad \Gamma_{1\alpha}^1=c_\bullet^{-1}\frac{\partial c_\bullet}{\partial y^\alpha},\\
&\Gamma_{11}^\alpha=-c_\bullet^{-1}g_\bullet^{\alpha\beta}\frac{\partial c_\bullet}{\partial y^\beta},\quad \Gamma_{11}^1=c_\bullet^{-1}\frac{\partial c_\bullet}{\partial s},
\end{split}
\end{equation}
where $\alpha,\beta\in \{2,3\}$.

Introduce the map $(z^0,z'):=(z^0,z^1,z^2,z^3)=\mathcal{F}_2(y^0=t,y^1=s,y^2,y^3)$, where
\[
z^0=\tau=\frac{1}{\sqrt{2}}(t-t_0+s),\quad z^1=r=\frac{1}{\sqrt{2}}(-t+t_0+s),\quad z^j=y^j,\, j=2,3.
\]
The Fermi coordinates $(z^0=\tau,z^1=r,z^2,z^3)$ near $\vartheta$ is given by $\mathcal{F}:\mathbb{R}^{1+3}\rightarrow \mathbb{R}\times\Omega'$, where $\mathcal{F}=\mathcal{F}_1\circ\mathcal{F}_2^{-1}$. Denote $\tau_\pm=\sqrt{2}(t_\pm-t_0)$.
Then on $\vartheta$ we have
\[
\overline{g}_\bullet\vert_\vartheta=2\mathrm{d}\tau\mathrm{d}r+\sum_{j=2}^3(\mathrm{d}z^j)^2\quad\text{and}\quad\frac{\partial \overline{g}_{\bullet,jk}}{\partial z^i}\Big\vert_\vartheta=0,~0\leq i,j,k\leq 3.
\]
\subsection{Gaussian beams}
With $\varrho$ as a large parameter, the Gaussian beam solutions have the asymptotic form as
\[
u_\varrho=\mathfrak{a}e^{\mathrm{i}\varrho\varphi},
\]
with 
\[
\varphi=\sum_{k=0}^N\varphi_k(\tau,z'),\quad \mathfrak{a}(\tau,z')=\chi\left(\frac{|z'|}{\delta}\right)\sum_{k=0}^{N+1}\varrho^{-k}\mathbf{a}_k(\tau,z'),\quad \mathbf{a}_k(\tau,z')=\sum_{j=0}^N\mathbf{a}_{k,j}(\tau,z'),
\]
which is compactly supported in a neighborhood of $\vartheta$,
\[
\mathcal{V}=\{(\tau,z')\in M'\vert \tau\in \left[\tau_--\frac{\epsilon}{\sqrt{2}},\tau_++\frac{\epsilon}{\sqrt{2}}\right],\,|z'|<\delta\}.
\]
Here $\delta>0$ is a small parameter. For each $j$, $\varphi_j$ and $\mathbf{a}_{k,j}$ are complex valued homogeneous polynomials of degree $j$ with respect to the variables $z^i$, $i=1,2,3$. The smooth function $\chi:\mathbb{R}\rightarrow [0,+\infty)$ satisfies $\chi(t)=1$ for $|t|\leq\frac{1}{4}$ and $\chi(t)=0$ for $|t|\geq \frac{1}{2}$. We refer to \cite{feizmohammadi2019recovery} for more details. Under the coordinates $(y^1,y^2,y^3)$ on $\Omega'$, we denote the (co)vector $\mathbf{a}_k=(a_{k1},a_{k2},a_{k3})$. The parameter $\delta$ is small enough such that $\mathfrak{a}\vert_{t=0}=\mathfrak{a}\vert_{t=T}=0$.\\

In the following calculations, we work under the coordinate system $(y^1,y^2,y^3)$ on the Riemannian manifold $(\Omega',g_E)$. That is, all the inner products and covariant derivatives are with respect to the Euclidean metric $g_E$. For simplicity of notations, we denote $\overline{g}=\overline{g}_\bullet$, $g=g_\bullet$ and $c=c_\bullet$, $\bullet=S,P$.
In a neighborhood of $\vartheta$, we can write (cf. \cite{uhlmann2021inverse})
\begin{equation}\label{eq_phase}
\rho\partial_t^2 u_\varrho-\nabla\cdot S^L(u_\varrho)=e^{\mathrm{i}\varrho\varphi}\left(\varrho^2\mathcal{I}_1+\sum_{k=0}^N\varrho^{1-k}\mathcal{I}_{k+2}+\mathcal{O}(\varrho^{-N})\right),
\end{equation}
where
\[
\mathcal{I}_1=-\rho(\partial_t\varphi)^2\mathbf{a}_0+(\lambda+\mu)\langle\mathbf{a}_0,\nabla\varphi\rangle\nabla\varphi+\mu|\nabla\varphi|^2\mathbf{a}_0,
\]
and
\[
\begin{split}
(\mathcal{I}_2)_i=&\rho(\partial^2_t\varphi)a_{0i}+2\rho\partial_t\varphi\partial_ta_{0i}+\mathrm{i}\rho(\partial_t\varphi)^2a_{1i}\\
&- c^{-2}\left(\partial_i(\lambda a_{0k}\varphi_{;\ell}g^{k\ell})+\lambda a_{0k}\varphi_{;\ell}g^{k\ell}\partial_mg_{ij}g^{jm}+\partial_m(\mu a_{0i}\varphi_{;j}+\mu a_{0j}\varphi_{;i})g^{jm}\right)\\
&-c^{-2}\left(\lambda a_{0k;\ell}g^{k\ell}\varphi_{;i}+\mu(a_{0i;j}+a_{j;i})\varphi_{;m}g^{jm}+\mathrm{i}\lambda a_{1k}\varphi_{;\ell}g^{k\ell}\varphi_{;i}+\mathrm{i}\mu(a_{1i}\varphi_{;j}\varphi_{;m}g^{jm}+\varphi_{;i}a_{1j}\varphi_{;m}g^{jm})\right)\\
&+\Gamma^n_{im}g^{mj}c^{-2}(\lambda a_{0k}\varphi_{;\ell}g^{k\ell}g_{nj}+\mu a_{0n}\varphi_{;j}+\mu a_{0j}\varphi_{;n})\\
&+\Gamma^n_{jm}g^{mj}c^{-2}(\lambda a_{0k}\varphi_{;\ell}g^{k\ell}g_{ni}+\mu a_{0n}\varphi_{;i}+\mu a_{0i}\varphi_{;n}).
\end{split}
\]

We will need to construct the phase function $\varphi$ and the amplitude $\mathfrak{a}$ such that
\begin{equation}\label{I1}
\frac{\partial^\Theta}{\partial z^\Theta}\mathcal{I}_k=0\text{ on } \vartheta
\end{equation}
for $\Theta=(0,\Theta_1,\Theta_2,\Theta_3)$ with $|\Theta|\leq N$ and $k=1,2,\cdots, N+2$.

\subsection{\textit{S}-waves}

For the construction of \textit{S}-waves, we note that $g=g_S$.
In order for $\mathcal{I}_1$ to vanish up to order $N$ on $\vartheta$ (cf. \eqref{I1} with $k=1$), we take $\varphi$ and $\mathbf{a}_0$ such that
\begin{equation}\label{eq_phase}
\frac{\partial^\Theta}{\partial y^\Theta}(\mathcal{S}\varphi)=0\text{ on } \vartheta,
\end{equation}
where $\mathcal{S}\varphi=\mu|\nabla\varphi|^2-\rho(\partial_t\varphi)^2$, and
\begin{equation}\label{eq_polarization}
\frac{\partial^\Theta}{\partial y^\Theta}\langle\mathbf{a}_0,\nabla\varphi\rangle=0\text{ on } \vartheta,
\end{equation}
for any $|\Theta|\leq N$. Taking $\Theta=0$ in \eqref{eq_polarization}, we conclude that $a_{01}\vert_{\vartheta}=0$.

For the construction of the phase function $\varphi$, we can take
\[
\varphi=\sum_{k=0}^N\varphi_k(\tau,z')
\]
such that
\[
\varphi_0(\tau,r,z'')=0,\quad \varphi_1(\tau,r,z'')=r,\quad \varphi_2(\tau,z')=\sum_{i,j=1}^3H_{ij}(\tau)z^iz^j,
\]
where $H$ is a symmetric matrix solving the Riccati equation
\begin{equation}\label{Ricatti}
\frac{\mathrm{d}}{\mathrm{d}\tau}H+HCH+D=0,\tau\in \left(\tau_--\frac{\epsilon}{2},\tau_++\frac{\epsilon}{2}\right),\quad H(\tau_0)=H_0,\text{ with }\Im H_0>0,
\end{equation}
where $C$, $D$ are matrices with $C_{11}=0$, $C_{ii}=2$, $i=2,3$, $C_{ij}=0$, $i\neq j$, $D_{ij}=\frac{1}{4}(\partial_{ij}^2\overline{g}^{11})$, $H_0$ is any given symmetric matrix with $\Im H_0>0$. Here $(\overline{g}^{ij})$ is the inverse of the Lorentzian metric $(\overline{g}_{ij})$ under the Fermi coordinates $(z^0,z^1,z^2,z^3)$, i.e., $\overline{g}_{ij}=\overline{g}(\frac{\partial}{\partial z^i},\frac{\partial}{\partial z^j})$. Then the equation \eqref{Ricatti} has a unique solution with $\Im(H(\tau))>0$ for all $\tau$.

For solving the Ricatti equation \eqref{Ricatti}, we take
\[
H(\tau)=Z(\tau)Y(\tau)^{-1},
\]
where $Z(\tau)$ and $Y(\tau)$ are solutions to the first order linear ODEs
\[
\begin{split}
\frac{\mathrm{d}}{\mathrm{d}\tau}Y=CZ,\quad Y(\tau_0)=Y_0,\\
\frac{\mathrm{d}}{\mathrm{d}\tau}Z=-DY,\quad Z(\tau_0)=H_0Y_0,
\end{split}
\]
where $Y_0$ is any non-degenerate matrix.
Here $Y(\tau)$ is non-degnerate for all $\tau$. Moreover, the following identity holds
\[
\det(\Im H(\tau))|\det(Y(\tau))|^2=c_0,
\]
where $c_0$ is a constant independent of $\tau$. For more discussions on the Ricatti equation, we refer to \cite{kachalov2001inverse}.
The higher order terms $\varphi_k$, $k=3,\cdots, N,$ can be constructed so that \eqref{eq_phase} is satisfied (cf. \cite{feizmohammadi2019timedependent,feizmohammadi2019recovery}).\\

Next consider the equation $(\mathcal{I}_2)_{i}=0$ for $i=2,3$. We have, for $\alpha=2,3$,
\[
\begin{split}
\left(\rho\partial^2_t\varphi-c_S^{-2}\mu\partial_s\varphi-c_S^{-2}\mu\sum_{\beta=2}^3\frac{\partial^2\varphi}{\partial y^\beta\partial y^\beta}\right)a_{0\alpha}\\
-\sqrt{2}\left(\rho\partial_ta_{0\alpha}+c_S^{-2}\mu\frac{\partial_sa_{0\alpha}}{\partial s}
\right)&\\
-\frac{1}{\sqrt{2}}c_S^{-2}\partial_s\mu a_{0\alpha}+\frac{1}{\sqrt{2}}\mu c_S^{-3}\frac{\partial c_S}{\partial s}a_{0\alpha}+\frac{1}{2}\mathrm{i}a_{1\alpha}(\rho-c_S^{-2}\mu)&\\
-c_S^{-2}(\lambda+\mu)\left(\frac{1}{\sqrt{2}}\frac{\partial a_{01}}{\partial y^\alpha}+\sum_{\beta=2}^3a_{0\beta}\frac{\partial^2\varphi}{\partial y^\alpha\partial y^\beta}\right)&=0.
\end{split}
\]
Notice that the above equation can be rewritten as
\begin{equation}\label{eq_a0}
\mathcal{T}a_{0\alpha}+\frac{1}{2}\mathrm{i}a_{1\alpha}(\rho-c_S^{-2}\mu)
-c_S^{-2}(\lambda+\mu)\left(\frac{1}{\sqrt{2}}\frac{\partial a_{01}}{\partial y^\alpha}+\sum_{\beta=2}^3a_{0\beta}\frac{\partial^2\varphi}{\partial y^\alpha\partial y^\beta}\right)=0,
\end{equation}
where
\[
\mathcal{T}=2\frac{\partial}{\partial \tau}+\left[\frac{1}{\mu}\frac{\partial\mu}{\partial\tau}-c_S^{-1}\frac{\partial c_S}{\partial \tau}+\frac{1}{\det(Y_S)}\frac{\partial \det(Y_S)}{\partial\tau}\right].
\]

%Since $\frac{\partial}{\partial y^a}\langle \mathbf{a}_{S,0},\nabla\varphi\rangle\vert_{\vartheta} =0$,
Consider the equation \eqref{eq_polarization} with $|\Theta|=1$, we have
\begin{equation}\label{eq_parallel_a}
\frac{1}{\sqrt{2}}\frac{\partial a_{01}}{\partial y^k}+\sum_{\beta=2}^3a_{0\beta}\frac{\partial^2\varphi}{\partial y^k\partial y^\beta}=0,
\end{equation}
for $k=1,2,3$. Insert \eqref{eq_parallel_a} (with $k=2,3$) into \eqref{eq_a0}, together with the fact $\rho=c_S^{-2}\mu$, we have
the following transport equation for $a_{0\alpha}$, $\alpha=2,3$,
\[
\mathcal{T}a_{0\alpha}=0.
\]
By solving the above transport equation, one can determine $a_{0\alpha}\vert_{\vartheta}$ for $\alpha=2,3$. In particular we can take
\[
a_{0\alpha}(\tau)=c_\alpha\det(Y_S(\tau))^{-1/2}c_S(\tau,0)^{-1/2}\rho(\tau,0)^{-1/2},
\]
where the constant $c_\alpha$ depends on the initial value of $a_{0\alpha}$. Using the equation \eqref{eq_parallel_a} again, we can now determine $\frac{\partial^\Theta a_{01}}{\partial y^\Theta}\Big\vert_{\vartheta}$ for $|\Theta|=1$. 

Noting that we have determined $\mathbf{a}_0$ and $\partial a_{01}$ on $\vartheta$ and using the equation $(\mathcal{I}_2)_{i=1}=0$, we end up with
\[
\begin{split}
\mathrm{i}\rho(\partial_t\varphi)^2a_{11}-\mathrm{i}c_S^{-2}(\lambda a_{1k}\varphi_{;\ell}g^{k\ell}\varphi_{;1}+\mu a_{11}\varphi_{;j}\varphi_{;m}g^{jm}+\varphi_{;1}a_{1j}\varphi_{;m}g^{jm})&\\
+2\rho\partial_t\varphi\partial_ta_{01}-c_S^{-2}(\lambda\partial_1a_{0k}\varphi_{;\ell}g^{k\ell}+(\mu\partial_ma_{01}\varphi_{;j}+\mu\partial_ma_{0j}\varphi_{;1})g^{jm})&\\
-c_S^{-2}(\lambda\partial_\ell a_{0k}g^{k\ell}\varphi_{;1}+\mu(\partial_ja_{01}+\partial_1a_{0j})\varphi_{;m}g^{jm})
&=F(\mathbf{a}_0),
\end{split}
\]
which further simplifies to
\begin{equation}\label{eq_ba}
\mathrm{i}(\rho-c_S^{-2}(\lambda+2\mu))a_{11}-\sqrt{2}c_S^{-2}(\lambda+\mu)(\partial_2a_{02}+\partial_3a_{03})=F(\mathbf{a}_0,\partial a_{01}).
\end{equation}
Notice that $\rho-c_S^{-2}(\lambda+2\mu)\neq 0$.
We will also use equation \eqref{eq_polarization} with $|\Theta|=2$, which can be written as
\begin{equation}\label{eq_polarization2}
\frac{1}{\sqrt{2}}\frac{\partial^2 a_{01}}{\partial y^m\partial y^k}+\frac{\partial a_{0j}}{\partial y^k}\frac{\partial^2\varphi}{\partial y^m\partial y^j}+\frac{\partial a_{0i}}{\partial y^m}\frac{\partial^2\varphi}{\partial y^k\partial y^i}=F(\mathbf{a}_0).
\end{equation}
%Note that we can not determine $a_{1\alpha}$ at this step. But we can determine $a_{11}$.
Next use the equation 
\[
\left(\frac{\partial^\Theta}{\partial y^\Theta}\mathcal{I}_2\right)_\alpha=0
\]
for $|\Theta|=1$ and $\alpha=2,3$ and substitute $a_{11}$ and $\partial^2a_{01}$ using \eqref{eq_ba} and \eqref{eq_polarization2},
%Together with \eqref{eq_ba} and
%\begin{equation}\label{eq_orth_derv}
%\frac{\partial^2 a_{01}}{\partial y^k\partial y^m}=F(\partial a_{02},\partial a_{03})
%\end{equation}
%which results from $\frac{\partial^2}{\partial y^k\partial y^m}\langle \mathbf{a}_{S,0},\nabla\varphi\rangle\vert_{\vartheta} =0$, one
we end up with a transport equation on $\vartheta$
\begin{equation}\label{transport_derv1}
\left(\frac{\partial}{\partial \tau}+A\right)\left(\begin{array}{c}\frac{\partial a_{02}}{\partial y^1}\\\frac{\partial a_{02}}{\partial y^2}\\\frac{\partial a_{02}}{\partial y^3}\\\frac{\partial a_{03}}{\partial y^1}\\\frac{\partial a_{03}}{\partial y^2}\\\frac{\partial a_{03}}{\partial y^3}\end{array}\right)=F(\mathbf{a}_0,\partial a_{01}).
\end{equation}
Here $A$ is a $6\times 6$ matrix depending on $\lambda,\mu,\rho$ and $\varphi$. Solving the above transport equation, one can determine $\frac{\partial^\Theta a_{0\alpha}}{\partial y^\Theta}\vert_{\vartheta}$ for any $|\Theta|=1$ and $\alpha=2,3$. Using \eqref{eq_ba} and \eqref{eq_polarization2} again, we can determine $\partial^2 a_{01}\vert_{\vartheta}$ and $a_{11}\vert_{\vartheta}$.

Now we have already determined $a_{01},\partial a_{01}, \partial^2 a_{01},a_{0\alpha},\partial a_{0\alpha}, a_{11}$ on $\vartheta$, $\alpha=2,3$.
Then we can use the equations
\[
(\mathcal{I}_3)_{i=2,3}=0,\quad (\partial\mathcal{I}_2)_{i=1}=0,\quad \partial^3\mathcal{I}_1=0,\quad (\partial^2\mathcal{I}_2)_{i=2,3}=0
\]
to determine $\partial^3a_{01},\partial^2a_{0\alpha},\partial a_{11},a_{1\alpha}$ for $\alpha=2,3$.

Continuing with the process, we can have \eqref{I1} satisfied, and finish our construction of Gaussian beam solutions for \textit{S}-waves.
% $(\mathcal{I}_3)_{i=2,3}=0$ to determine $a_{12}$ and $a_{13}$. Using $(\frac{\partial}{\partial z^\Theta}\mathcal{I}_2)_{i=1}=0$ to determine $\frac{\partial a_{11}}{\partial z^\Theta}$ for $|\Theta|=1$.

\subsection{\textit{P}-waves}\label{Pgaussian}
For the construction of \textit{P}-waves, we take $\varphi$ such that
\[
\frac{\partial^\Theta}{\partial y^\Theta}(\mathcal{S}\varphi)=0\text{ on } \vartheta,
\]
for any $|\Theta|\leq N$, where $\mathcal{S}\varphi=(\lambda+2\mu)|\nabla\varphi|^2-\rho(\partial_t\varphi)^2$. The phase function $\varphi$ can be constructed similarly as for the \textit{S}-waves. Now we denote $g=g_S$.

Take
\[
\mathbf{a}_0=A_P\nabla\varphi.
\]
Then the equation \eqref{I1} is satisfied for $k=1$ since $\mathcal{I}_1=(\mathcal{S}\varphi)A_P\nabla\varphi$.

First consider the equation $(\mathcal{I}_2)_{i}=0$ for $i=1$. We obtain
\[
\begin{split}
\frac{1}{\sqrt{2}}\left(\rho \partial^2_t\varphi-c_P^{-2}(\lambda+2\mu)\partial_s^2\varphi-c_P^{-2}(\lambda+2\mu)\sum_{\alpha=2}^3\frac{\partial^2\varphi}{\partial y^\alpha\partial y^\alpha}\right)A_P&\\
-\left(\rho\partial_t A_P+c_P^{-2}(\lambda+2\mu)\partial_sA_P\right)&\\
+\frac{1}{2}(\lambda+2\mu)c_P^{-3}\frac{\partial c_P}{\partial s}A_P-\frac{1}{2}c_P^{-2}\partial_s(\lambda+2\mu)A_P
+\mathrm{i}\frac{1}{2}b_1[\rho-c_P^{-2}(\lambda+2\mu)]&\\
-\sqrt{2}\left(c_P^{-2}(\lambda+2\mu)\frac{\partial^2\varphi}{\partial s^2}+\rho\frac{\partial^2\varphi}{\partial s\partial t }\right)A_P&=0.
\end{split}
\] 
We have proved in \cite{uhlmann2021inverse} that
\[
c_P^{-2}(\lambda+2\mu)\frac{\partial^2\varphi}{\partial s^2}+\rho\frac{\partial^2\varphi}{\partial s\partial t }=\sqrt{2}\rho\frac{\partial}{\partial\tau}(\partial_s\varphi)=0,\quad \text{on }\vartheta,
\]
since $\partial_s\varphi=\frac{1}{\sqrt{2}}$ is constant along $\vartheta$.
Using the fact that $\rho-c_P^{-2}(\lambda+2\mu)=0$ , the above equation can be rewritten as a transport equation
\[
\mathcal{T}A_P=0
\]
where
\[
\mathcal{T}=2\frac{\partial}{\partial \tau}+\left[\frac{1}{\lambda+2\mu}\frac{\partial(\lambda+2\mu)}{\partial\tau}-c_P^{-1}\frac{\partial c_P}{\partial \tau}+\frac{1}{\det(Y_P)}\frac{\partial \det(Y_P)}{\partial\tau}\right].
\]
Then $A_P\vert_{\vartheta}$ can be determined by solving this transport equation. In particular, we can take
\[
A_P(\tau)=c\det(Y_P(\tau))^{-1/2}c_P(\tau,0)^{-1/2}\rho(\tau,0)^{-1/2},
\]
where the constant $c$ depends on the initial value of $A_P$.

Using the equation $(\mathcal{I}_2)_i=0$ with $i=2,3$, we end up with
\begin{equation}\label{eq_ba2}
\mathrm{i}(\rho-c_P^{-2}\mu)a_{1\alpha}-c_P^{-2}(\lambda+\mu)\frac{\partial A_P}{\partial y^\alpha}=F(A_P),
\end{equation}
for $\alpha=2,3$.
Notice that $\rho-c_P^{-2}\mu\neq 0$. Next use the equation 
\[
\left(\frac{\partial^\Theta}{\partial y^\Theta}\mathcal{I}_2\right)_\alpha=0
\]
for $|\Theta|=1$ and $\alpha=1$ and substitute $a_{1\alpha}$, $\alpha=2,3$ using \eqref{eq_ba2}, we end up with a transport equation
\begin{equation}\label{transport_derv2}
\left(\frac{\partial}{\partial \tau}+A\right)\left(\begin{array}{c}\frac{\partial A_P}{\partial y^1}\\\frac{\partial A_P}{\partial y^2}\\\frac{\partial A_P}{\partial y^3}\end{array}\right)=F(A_P).
\end{equation}
Solving the above transport equation, one can determine $\frac{\partial^\Theta A_P}{\partial y^\Theta}\vert_{\vartheta}$ for any $|\Theta|=1$. Using \eqref{eq_ba2} again, we can determine $a_{1\alpha}\vert_\vartheta$ for $\alpha=2,3$.

Similar as the construction for \textit{S}-waves above, we can finish the construction of the \textit{P}-wave Gaussian beam solutions such that \eqref{I1} is satisfied.
\section{Gaussian beams with reflections}\label{gaussianreflection}
We will discuss the reflection of Gaussian beams at the boundary in this section. The reflection condition we considered here is the traction-free boundary condition, which is natural in practice. The reflection of (real) geometric optics solutions is analyzed in \cite{stefanov2021transmission}.  We only give detailed characterization for the case where the boundary is locally flat, and non-flat case will not require much more work by flattening it \cite{stefanov2021transmission}. Assume near a fixed point $x_0$, $\partial\Omega$ is locally expressed as $x_3=0$ and $\Omega$ is represented as $\{x_3<0\}$. We are looking for $u_\varrho$ as a sum of two solutions $u_\varrho=u_\varrho^++u_\varrho^-$, where $u_\varrho^+$ is the incident wave and $u_\varrho^-$ is the reflected wave generated by $u_\varrho^+$ hitting the boundary. We remark here that the reflected wave $u_\varrho^-$ can undergo further reflections, and after multiple reflections the behavior of the waves would become very complicated, but we only need to discuss single reflection.
\subsection{\textit{P}- incident waves}
We first consider the case when the incident wave is a \textit{P}-wave. The reflected wave is a combination of \textit{P}- and \textit{S}- waves.
Assume the incident wave
\[
u^+_\varrho=\sum_{k=0}^{N+1}\varrho^{-k}\mathbf{a}_{P,k}^+\chi^+e^{\mathrm{i}\varrho\varphi_P^+}
\]
is a Gaussian beam traveling along the null-geodesic $\vartheta^+$ as constructed as in Section \ref{Pgaussian}, and $\vartheta^+$ intersects with the boundary at the point $p\in(0,T)\times\partial\Omega$.
We seek for reflected waves as
\begin{equation}\label{reflectedwaves}
u^-_\varrho=u^-_{P,\varrho}+u^-_{S,\varrho}:=\sum_{k=0}^{N+1}\varrho^{-k}\mathbf{a}_{P,k}^-\chi^-_Pe^{\mathrm{i}\varrho\varphi_P^-}+\sum_{k=0}^{N+1}\varrho^{-k}\mathbf{a}_{S,k}^-\chi^-_Se^{\mathrm{i}\varrho\varphi_S^-},
\end{equation}
where $u^-_{P,\varrho},u^-_{S,\varrho}$ are also Gaussian beam solutions representing \textit{P}- and \textit{S}- waves respectively.
Here $\mathbf{a}^\pm_{\bullet,k}=(a^\pm_{\bullet,k1},a^\pm_{\bullet,k2},a^\pm_{\bullet,k3})$. The reflected wave $u^-_\varrho$ is generated such that $u_\varrho=u^+_\varrho+u^-_\varrho$ satisfies the Neumann boundary condition 
\[
\mathcal{N}_{\lambda,\mu}u_\varrho:=\nu\cdot S^L(x,u_\varrho)\vert_{\mathbb{R}\times\partial\Omega}=0,
\]
 up to order $N$, at the point $p$. \\

Assume $\nabla \varphi_P^+(p)=\xi^+_P=(\xi^+_{P,1},\xi^+_{P,2},\xi^+_{P,3})$, $\nabla \varphi_P^-(p)=\xi^-_P=(\xi^-_{P,1},\xi^-_{P,2},\xi^-_{P,3})$, $\nabla \varphi_S^-(p)=\xi^-_S=(\xi^-_{S,1},\xi^-_{S,2},\xi^-_{S,3})$, where $|\xi^+_P|=|\xi^-_P|=\frac{1}{c_P(p)}$ and $|\xi^-_S|=\frac{1}{c_S(p)}$, $\xi^+_{P,3}>0$. By Snell's law, we have
\[
\xi^+_{P,1}=\xi^-_{P,1}=\xi^-_{S,1},\quad \xi^+_{P,2}=\xi^-_{P,2}=\xi^-_{S,2},\quad \xi^-_{P,3}=-\xi^-_{P,3},\quad \xi^-_{S,3}=-\sqrt{c_S^{-2}-(\xi^+_{P,1})^2-(\xi^+_{P,2})^2}.
\]
To simplify the notation, we denote
\[
\xi^+_P=(\xi_1,\xi_2,\xi_3), \quad \xi^-_P=(\xi_1,\xi_2,-\xi_3),\quad \xi^-_S=(\xi_1,\xi_2,-\xi_{S,3}).
\]

The phase functions $\varphi^-_P$ and $\varphi^-_S$ can be constructed as in previous section such that $
\varphi^-_P\vert_{\mathbb{R}\times\partial\Omega}=\varphi^-_S\vert_{\mathbb{R}\times\partial\Omega}=\varphi^+_P\vert_{\mathbb{R}\times\partial\Omega}
$ at $p$ up to order $N$.
Then note that we have the following asymptotics
\[
\mathcal{N}_{\lambda,\mu}u_\varrho=e^{\mathrm{i}\varrho\varphi}(\varrho\mathcal{N}_0+\mathcal{N}_1+\varrho^{-1}\mathcal{N}_{2}+\cdots),
\]
where $\varphi=\varphi^+_P\vert_{\mathbb{R}\times\partial\Omega}$.
We need
\begin{equation}\label{NeumannCondition}
\partial^\Theta\mathcal{N}_k=0 \text{ at } p,\quad \text{for } k\geq 0,\, |\Theta|\leq N.
\end{equation}

By calculation, we obtain
\[
\begin{split}
\mathcal{N}_0=&\left(\begin{array}{ccc}\mu \partial_3\varphi_P^+&0 & \mu \partial_1\varphi_P^+\\
0 &\mu \partial_3\varphi_P^+ &\mu \partial_2\varphi_P^+\\
\lambda \partial_1\varphi_P^+ &\lambda \partial_2\varphi_P^+&(\lambda+2\mu) \partial_3\varphi_P^+
\end{array}\right)\left(\begin{array}{c}
a_{P,01}^+\\
a_{P,02}^+\\
a_{P,03}^+
\end{array}\right)\\
&+\left(\begin{array}{ccc}\mu \partial_3\varphi_P^-&0 & \mu \partial_1\varphi_P^-\\
0 &\mu \partial_3\varphi_P^- &\mu \partial_2\varphi_P^-\\
\lambda \partial_1\varphi_P^- &\lambda \partial_2\varphi_P^-&(\lambda+2\mu) \partial_3\varphi_P^-
\end{array}\right)\left(\begin{array}{c}
a_{P,01}^-\\
a_{P,02}^-\\
a_{P,03}^-
\end{array}\right)\\
&+\left(\begin{array}{ccc}\mu \partial_3\varphi_{S}^-&0 & \mu \partial_1\varphi_{S}^-\\
0 &\mu \partial_3\varphi_{S}^- &\mu \partial_2\varphi_{S}^-\\
\lambda \partial_1\varphi_{S}^- &\lambda \partial_2\varphi_{S}^-&(\lambda+2\mu) \partial_3\varphi_{S}^-
\end{array}\right)\left(\begin{array}{c}
a_{S,01}^-\\
a_{S,02}^-\\
a_{S,03}^-
\end{array}\right).
\end{split}
\]
%Also we need
%\[
%\left(\begin{array}{ccc}\mu \partial_3\varphi_P^+&0 & \mu \partial_1\varphi_P^+\\
%0 &\mu \partial_3\varphi_P^+ &\mu \partial_2\varphi_P^+\\
%\lambda \partial_1\varphi_P^+ &\lambda \partial_2\varphi_P^+&(\lambda+2\mu) \partial_3\varphi_P^+
%\end{array}\right)\left(\begin{array}{c}
%a_{P,01}^+\\
%a_{P,02}^+\\
%a_{P,03}^+
%\end{array}\right)
%\]
Using the fact
\[
\mathbf{a}_{P,0}^-=A_P^-\nabla\varphi_P^-,\quad \mathbf{a}_{P,0}^+=A_P^+\nabla\varphi_P^+,
\]
we have
\begin{equation}\label{eq_sys3}
\begin{split}
\mathcal{N}_0=&\left(\begin{array}{ccc}\mu \partial_3\varphi_P^-&0 & \mu \partial_1\varphi_P^-\\
0 &\mu \partial_3\varphi_P^- &\mu \partial_2\varphi_P^-\\
\lambda \partial_1\varphi_P^- &\lambda \partial_2\varphi_P^-&(\lambda+2\mu) \partial_3\varphi_P^-
\end{array}\right)\left(\begin{array}{c}
\partial_1\varphi_P^-\\
\partial_2\varphi_P^-\\
\partial_3\varphi_P^-
\end{array}\right)A_P^-\\
&+\left(\begin{array}{ccc}\mu \partial_3\varphi_{S}^-&0 & \mu \partial_1\varphi_{S}^-\\
0 &\mu \partial_3\varphi_{S}^- &\mu \partial_2\varphi_{S}^-\\
\lambda \partial_1\varphi_{S}^- &\lambda \partial_2\varphi_{S}^-&(\lambda+2\mu) \partial_3\varphi_{S}^-
\end{array}\right)\left(\begin{array}{c}
a_{S,01}^-\\
a_{S,02}^-\\
a_{S,03}^-
\end{array}\right)\\
&+\left(\begin{array}{ccc}\mu \partial_3\varphi_P^+&0 & \mu \partial_1\varphi_P^+\\
0 &\mu \partial_3\varphi_P^+ &\mu \partial_2\varphi_P^+\\
\lambda \partial_1\varphi_P^+ &\lambda \partial_2\varphi_P^+&(\lambda+2\mu) \partial_3\varphi_P^+
\end{array}\right)\left(\begin{array}{c}
\partial_1\varphi_P^+\\
\partial_2\varphi_P^+\\
\partial_3\varphi_P^+
\end{array}\right)A_P^+.
\end{split}
\end{equation}
Considering also the identity
\begin{equation}\label{eq_orth_a}
\langle\mathbf{a}_{S,0}^-,\nabla\varphi_S^-\rangle=0,
\end{equation}
we have the following equations at point $p$:
\[
\begin{split}
\left(\begin{array}{ccc}-\mu \xi_3&0 & \mu \xi_1\\
0 &-\mu \xi_3 &\mu\xi_2\\
\lambda \xi_1 &\lambda \xi_2&-(\lambda+2\mu) \xi_3
\end{array}\right)\left(\begin{array}{c}
\xi_1\\
\xi_2\\
-\xi_3
\end{array}\right)A_P^-(p)\\
+\left(\begin{array}{ccc}-\mu \xi_{S,3}&0 & \mu \xi_1\\
0 &-\mu \xi_{S,3} &\mu \xi_2\\
\lambda \xi_1 &\lambda \xi_2&-(\lambda+2\mu) \xi_{S,3}
\end{array}\right)\left(\begin{array}{c}
a_{S,01}^-(p)\\
a_{S,02}^-(p)\\
a_{S,03}^-(p)
\end{array}\right)\\
=-\left(\begin{array}{ccc}\mu \xi_3&0 & \mu \xi_1\\
0 &\mu \xi_3 &\mu \xi_2\\
\lambda \xi_1 &\lambda \xi_2&(\lambda+2\mu) \xi_3
\end{array}\right)\left(\begin{array}{c}
\xi_1\\
\xi_2\\
\xi_3
\end{array}\right)A_{P}^+(p).
\end{split}
\]
and
\[
(\xi_1,\xi_2,-\xi_{S,3})\left(\begin{array}{c}
a_{S,01}^-(p)\\
a_{S,02}^-(p)\\
a_{S,03}^-(p)
\end{array}\right)=0.
\]
The equations above can be formulated into the following linear system for $(A_P^-,a_{S,01}^-,a_{S,02}^-,a_{S,03}^-)$ at $p$:
\begin{equation}\label{linsys4M}
\begin{split}
M_P(\xi)\left(\begin{array}{c}
A_P^-(p)\\
a_{S,01}^-(p)\\
a_{S,02}^-(p)\\
a_{S,03}^-(p)
\end{array}\right):=&\left(
\begin{array}{cccc}
-2\mu\xi_1\xi_3 & -\mu\xi_{S,3} &0 & \mu\xi_1\\
-2\mu\xi_2\xi_3 & 0 & -\mu\xi_{S,3} & \mu\xi_2\\
\rho-2\mu(\xi_1^2+\xi_2^2) &\lambda\xi_1 &\lambda\xi_2 & -(\lambda+2\mu)\xi_{S,3}\\
0 & \xi_1 &\xi_2 &-\xi_{S,3}
\end{array}
\right)
\left(\begin{array}{c}
A_P^-(p)\\
a_{S,01}^-(p)\\
a_{S,02}^-(p)\\
a_{S,03}^-(p)
\end{array}\right)\\
=&-\left(\begin{array}{c}
-2\mu\xi_1\xi_3A_{P}^+(p)\\
-2\mu\xi_2\xi_3A_{P}^+(p)\\
\rho-2\mu(\xi_1^2+\xi_2^2)A_{P}^+(p)\\
0
\end{array}\right).
\end{split}
\end{equation}
\begin{lemma}\label{lemma1}
The matrix $M_P(\xi)$ is invertible.
\end{lemma}
\begin{proof}
To see that the matrix $M_P(\xi)$ is invertible, without loss of generality, we assume $\xi_2=0$. Consider the homogeneous equation
\begin{equation}\label{MatrixM}
M(\xi)\left(\begin{array}{c}
A_P\\
a_{S,1}\\
a_{S,2}\\
a_{S,3}
\end{array}\right)=0.
\end{equation}
By the last equation, we can assume that
\[
\left(\begin{array}{c}
a_{S,1}\\
a_{S,2}\\
a_{S,3}
\end{array}\right)=A_{SH}\mathbf{e}_{SH}+A_{SV}\mathbf{e}_{SV}
\]
where
\[
\mathbf{e}_{SH}=\left(\begin{array}{c}
0\\
-c_S^{-1}\\
0
\end{array}\right)
\]
\[
\mathbf{e}_{SV}
=\left(\begin{array}{c}
\xi_{S,3}\\
0\\
\xi_1
\end{array}\right).
\]
The homogeneous equation simplifies into
\[
M'_{P}(\xi)(A_P,A_{SV},A_{SH})^T=0,
\]
where
\[
M_{P}'(\xi)=\left(
\begin{array}{ccc}
-2\mu\xi_1\xi_3 &-\rho+2\mu\xi_1^2&0\\
0 &0 &-\mu c_S^{-1}\xi_{S,3}\\
\rho-2\mu\xi_1^2&-2\mu\xi_1\xi_{S,3}& 0
\end{array}
\right).
\]
One immediately obtain that $A_{SH}=0$. Then we only need to consider the matrix
\[
\left(
\begin{array}{cc}
-2\mu\xi_1\xi_3 &\mu(\xi_1^2-\xi_{S,3}^2)\\
-\mu(\xi_1^2-\xi_{S,3}^2)&-2\mu\xi_1\xi_{S,3}
\end{array}
\right),
\]
whose determinant is
\[
\mu^2(4\xi_1^2\xi_3\xi_{S,3}+\xi_1^4-2\xi_1^2\xi_{S,3}^2+\xi_{S,3}^4)>0.
\]
This is because that $\xi_1^2\xi_3\xi_{S,3}<\xi_1^2\xi_{S,3}^2$, which in turn yields
\[
0<4\xi_1^2\xi_3\xi_{S,3}-2\xi_1^2\xi_{S,3}^2<2\xi_1^2\xi_{S,3}^2\leq \xi_1^4+\xi_{S,3}^4.
\]
This shows that the equation \eqref{MatrixM} has has only zero solutions, and therefore $M(\xi)$ is invertible.
\end{proof}
By the above lemma, we can solve the linear system \eqref{linsys4M} to determine $(A_P^-,a_{S,01}^-,a_{S,02}^-,a_{S,03}^-)$ at the point $p$.

Next, we determine the tangential derivatives of $(A_P^-,a_{S,01}^-,a_{S,02}^-,a_{S,03}^-)$ at $p$. Use the fact that $\mathcal{N}_0=0$ needs to be satisfied up to order $N$ at $p$, that is $\partial_{t,x}^\alpha\mathcal{N}_1(p)=0$, for $\alpha=(\alpha_0,\alpha_1,\alpha_2,0)$, $|\alpha|\leq N$. Let us first consider, for example, $(\partial_1A_P^-,\partial_1a_{S,01}^-,\partial_1a_{S,02}^-,\partial_1a_{S,03}^-)$. Taking the first order derivative in $x_1$ of \eqref{eq_sys3} and \eqref{eq_orth_a}, we obtain a linear system
\[
M(\xi)\left(\begin{array}{c}\partial_1A_P^-(p)\\\partial_1a_{S,01}^-(p)\\\partial_1a_{S,02}^-(p)\\\partial_1a_{S,03}^-(p)\end{array}\right)=F(A_P^-(p),\mathbf{a}_{S,0}^-(p)).
\]
Here and below we suppress the dependence of $F$ on $\mathbf{a}^+_{P,k}$ and $\varphi^+_P,\varphi^-_P,\varphi^-_S$ (as well as their derivates).
Because the matrix $M(\xi)$ is invertible, we can determine $(\partial_1A_P^-,\partial_1a_{S,01}^-,\partial_1a_{S,02}^-,\partial_1a_{S,03}^-)$ at the point $p$. Continue with this process, one can determine $(\partial^\alpha A_P^-,\partial^\alpha a_{S,01}^-,\partial^\alpha a_{S,02}^-,\partial^\alpha a_{S,03}^-)(p)$ for any $\alpha=(\alpha_0,\alpha_1,\alpha_2,0)$, $|\alpha|\leq N$.\\

Next, consider the determination of $\mathbf{a}_{P,1}^-$ and $\mathbf{a}_{S,1}^-$ at the point $p$. Notice that
\begin{equation}
\begin{split}
\mathcal{N}_1=&\left(\begin{array}{ccc}\mu \partial_3\varphi_P^-&0 & \mu \partial_1\varphi_P^-\\
0 &\mu \partial_3\varphi_P^- &\mu \partial_2\varphi_P^-\\
\lambda \partial_1\varphi_P^- &\lambda \partial_2\varphi_P^-&(\lambda+2\mu) \partial_3\varphi_P^-
\end{array}\right)\left(\begin{array}{c}
a_{P,11}^-\\
a_{P,12}^-\\
a_{P,13}^-
\end{array}\right)\\
&+\left(\begin{array}{ccc}\mu \partial_3\varphi_{S}^-&0 & \mu \partial_1\varphi_{S}^-\\
0 &\mu \partial_3\varphi_{S}^- &\mu \partial_2\varphi_{S}^-\\
\lambda \partial_1\varphi_{S}^- &\lambda \partial_2\varphi_{S}^-&(\lambda+2\mu) \partial_3\varphi_{S}^-
\end{array}\right)\left(\begin{array}{c}
a_{S,11}^-\\
a_{S,12}^-\\
a_{S,13}^-
\end{array}\right)\\
&+F(\mathbf{a}^-_{P,0},\mathbf{a}^-_{S,0},\partial\mathbf{a}^-_{P,0},\partial\mathbf{a}^-_{S,0})
\end{split}
\end{equation}
Note that the full first-order derivatives of $\mathbf{a}^-_{P,0},\mathbf{a}^-_{S,0}$ at point $p$, $\partial\mathbf{a}^-_{P,0}(p),\partial\mathbf{a}^-_{S,0}(p)$ (not only the derivatives in $(t,x_1,x_2)$), are determined since the transport equations for $A_P^-$ and $\mathbf{a}^-_{S,0}$ are satisfied.
The equation $\mathcal{N}_1(p)=0$ can be rewritten as 
\begin{equation}\label{linsys_a1}
\begin{split}
\left(\begin{array}{ccc}-\mu \xi_3&0 & \mu \xi_1\\
0 &-\mu \xi_3 &\mu\xi_2\\
\lambda \xi_1 &\lambda \xi_2&-(\lambda+2\mu) \xi_3
\end{array}\right)\left(\begin{array}{c}
a_{P,11}^-(p)\\
a_{P,12}^-(p)\\
a_{P,13}^-(p)
\end{array}\right)\\
+\left(\begin{array}{ccc}-\mu \xi_{S,3}&0 & \mu \xi_1\\
0 &-\mu \xi_{S,3} &\mu \xi_2\\
\lambda \xi_1 &\lambda \xi_2&-(\lambda+2\mu) \xi_{S,3}
\end{array}\right)\left(\begin{array}{c}
a_{S,11}^-(p)\\
a_{S,12}^-(p)\\
a_{S,13}^-(p)
\end{array}\right)
=F,
\end{split}
\end{equation}
where $F$ have already been computed in the previous steps.
Now, we have $3$ equations for the $6$ unknowns. The extra $3$ equations come from equations \eqref{eq_ba} and \eqref{eq_ba2}, which can be rewritten as
\begin{equation}\label{eq_aS1}
(\xi_1,\xi_2,-\xi_{S,3})\cdot (a_{S,11}^-(p),a_{S,12}^-(p),a_{S,13}^-(p))=F_1(\mathbf{a}^-_{P,0,}(p),\partial \mathbf{a}^-_{P,0,}(p),\mathbf{a}^-_{S,0}(p),\partial\mathbf{a}^-_{S,0,}(p)),
\end{equation}
and
\begin{equation}\label{eq_11}
\begin{split}
(\xi_2,-\xi_1,0)\cdot (a_{P,11}^-(p),a_{P,12}^-(p),a_{P,13}^-)(p)=F_2(\mathbf{a}^-_{P,0}(p),\partial \mathbf{a}^-_{P,0}(p),\mathbf{a}^-_{S,0}(p),\partial\mathbf{a}^-_{S,0}(p)),\\
(\xi_1\xi_3,\xi_2\xi_3,\xi_1^2+\xi_2^2)\cdot (a_{P,11}^-(p),a_{P,12}^-(p),a_{P,13}^-(p))=F_3(\mathbf{a}^-_{P,0}(p),\partial \mathbf{a}^-_{P,0}(p),\mathbf{a}^-_{S,0}(p),\partial\mathbf{a}^-_{S,0}(p)).
\end{split}
\end{equation}
Using \eqref{eq_11}, we can reduce the linear system \eqref{linsys_a1} to
\[
\begin{split}
\left(\begin{array}{ccc}-\mu \xi_3&0 & \mu \xi_1\\
0 &-\mu \xi_3 &\mu\xi_2\\
\lambda \xi_1 &\lambda \xi_2&-(\lambda+2\mu) \xi_3
\end{array}\right)\left(\begin{array}{c}
\xi_1\\
\xi_2\\
-\xi_3
\end{array}\right)A_{P,1}^-(p)\\
+\left(\begin{array}{ccc}-\mu \xi_{S,3}&0 & \mu \xi_1\\
0 &-\mu \xi_{S,3} &\mu \xi_2\\
\lambda \xi_1 &\lambda \xi_2&-(\lambda+2\mu) \xi_{S,3}
\end{array}\right)\left(\begin{array}{c}
a_{S,11}^-(p)\\
a_{S,12}^-(p)\\
a_{S,13}^-(p)
\end{array}\right)
=F,
\end{split}
\]
where $A_{P,1}^-(p)=(a_{P,11}^-(p),a_{P,12}^-(p),a_{P,13}^-(p))\cdot(\xi_1,\xi_2,-\xi_3)$. Together with \eqref{eq_aS1}, we now have a system
\[
M_P(\xi)(A_{P,1}^-(p),a_{S,11}^-(p),a_{S,12}^-(p),a_{S,13}^-(p))^T=F.
\]
By Lemma \ref{lemma1}, the matrix $M_P(\xi)$ is invertible.
Then, $a_{P,11}^-,a_{P,12}^-,a_{P,13}^-,a_{S,11}^-,a_{S,12}^-,a_{S,13}^-$ at point $p$ can be determined.

To determine the first order derivative (for example in $x_1$) of $\mathbf{a}_{P,1}^-$ and $\mathbf{a}_{S,1}^-$ at $p$, we take the $\partial_{x_1}$ derivative of \eqref{linsys_a1}, and also use the equations
\[
\begin{split}
(\xi_1,\xi_2,-\xi_{S,3})\cdot (\partial_{x_1}a_{S,11}^-(p),\partial_{x_1}a_{S,12}^-(p),\partial_{x_1}a_{S,13}^-(p))=F_1,\\
(\xi_2,-\xi_1,0)\cdot (\partial_{x_1}a_{P,11}^-(p),\partial_{x_1}a_{P,12}^-(p),\partial_{x_1}a_{P,13}^-(p))=F_2,\\
(\xi_1\xi_3,\xi_2\xi_3,\xi_1^2+\xi_2^2)\cdot (\partial_{x_1}a_{P,11}^-(p),\partial_{x_1}a_{P,12}^-(p),\partial_{x_1}a_{P,13}^-(p))=F_3,
\end{split}
\]
which come from $\partial_{x_1}(\mathcal{I}_2)_i(p)=0$ (more precisely, $i=2,3$ for \textit{P}-wave, and $i=1$ for \textit{S}-wave).
Consequently, we end up with a linear system as
\[
\left(\begin{array}{cccccc}-\mu \xi_3&0 & \mu \xi_1 &-\mu \xi_{S,3}&0 & \mu \xi_1\\
0 &-\mu \xi_3 &\mu\xi_2 &0 &-\mu \xi_{S,3} &\mu \xi_2\\
\lambda \xi_1 &\lambda \xi_2&-(\lambda+2\mu) \xi_3 &\lambda \xi_1 &\lambda \xi_2&-(\lambda+2\mu) \xi_{S,3}\\
0& 0&0& \xi_1 &\xi_2&-\xi_{S,3}\\
\xi_2&-\xi_1& 0 &0 &0 &0\\
\xi_1\xi_3&\xi_2\xi_3&\xi_1^2+\xi_2^2 & 0 & 0 & 0
\end{array}\right)
\left(\begin{array}{c}\partial_1a_{P,11}^-(p)\\
\partial_1a_{P,12}^-(p)\\
\partial_1a_{P,13}^-(p)\\
\partial_1a_{S,11}^-(p)\\
\partial_1a_{S,12}^-(p)\\
\partial_1a_{S,13}^-(p)\end{array}\right)=F.
\]
%\[
%M(\xi)(\partial_1A_{P,1}^-,\partial_1a_{S,11}^-,\partial_1a_{S,12}^-,\partial_1a_{S,13}^-)^T=F,
%\]
By similar consideration as above, this linear system is solvable, and thus we can determine $\partial_1a_{P,11}^-,\partial_1a_{P,12}^-,\partial_1a_{P,13}^-,\partial_1a_{S,11}^-,\partial_1a_{S,12}^-,\partial_1a_{S,13}^-$ at $p$.

Continuing with this process, we can determine $\partial^\alpha\mathbf{a}^-_{P,k}(p)$ and $\partial^\alpha\mathbf{a}^-_{S,k}(p)$ for any $k\leq N$ and $|\alpha|\leq N$. Now we finish the construction of the reflected waves.
\subsection{\textit{S}- incident waves}
In this section, we assume that the incident wave is \textit{S}-wave, i.e.,
\[
u^+_\varrho=e^{\mathrm{i}\varrho\varphi_S^+}\chi\left(\sum_{k=0}^{N+1}\varrho^{-k}\mathbf{a}_{S,k}^+\right).
\]
We construct reflected waves given by \eqref{reflectedwaves}.
Denote
\[
\xi^+_S:=\xi=(\xi_1,\xi_2,\xi_3), \quad \xi^-_S=(\xi_1,\xi_2,-\xi_3),\quad \xi^-_P=(\xi_1,\xi_2,-\xi_{P,3}).
\]
where 
\begin{equation}\label{xiP3}
\xi_{P,3}=\sqrt{c_P^{-2}-\xi_1^2-\xi_2^2}.
\end{equation}
Then the phase functions satisfy 
\[
\nabla \varphi_S^+(p)=\xi^+_S, \quad\nabla \varphi_P^-(p)=\xi^-_P, \quad\nabla \varphi_S^-(p)=\xi^-_S.
\] 

First we consider the important case, for which $\xi_1^2+\xi_2^2< c_P^{-2}$, so that $\xi_{P,3}$ is real and there is no evanescent wave. The reflected \textit{P}- and \textit{S}- waves are both progressing waves. The equation $\mathcal{N}_0(p)=0$ would give a system
\[
M_S(\xi)\left(\begin{array}{c}
A_P^-(p)\\
a_{S,01}^-(p)\\
a_{S,02}^-(p)\\
a_{S,03}^-(p)
\end{array}\right)=F,
\]
where
\[
M_S(\xi)=\left(
\begin{array}{cccc}
-2\mu\xi_1\xi_{P,3} & -\mu\xi_{3} &0 & \mu\xi_1\\
-2\mu\xi_2\xi_{P,3} & 0 & -\mu\xi_{3} & \mu\xi_2\\
\rho-2\mu(\xi_1^2+\xi_2^2) &\lambda\xi_1 &\lambda\xi_2 & -(\lambda+2\mu)\xi_{3}\\
0 & \xi_1 &\xi_2 &-\xi_{3}
\end{array}
\right).
\]
The matrix $M_S(\xi)=M_P((\xi_1,\xi_2,\xi_{P,3}))$ is invertible. Therefore, similar as in previous section, we can determine $\partial^\alpha\mathbf{a}^-_{P,k}(p)$ and $\partial^\alpha\mathbf{a}^-_{S,k}(p)$ for any $k\leq N$ and $|\alpha|\leq N$.\\

\noindent\textbf{Evanescent waves.} 
Now we consider the case 
$
\xi_1^2+\xi_2^2>c_P^{-2}.
$
In this case, $\xi_{P,3}$ is taken by
\[
\xi_{P,3}=\mathrm{i}\sqrt{\xi_1^2+\xi_2^2-c_P^{-2}}.
\]
To be consistent with the formula \eqref{xiP3}, we can choose the square root function $\sqrt{\,\cdot\,}$ such that $\Im\sqrt{z}>0$ for $z\in\mathbb{C}\setminus [0,+\infty)$. 
Then $\Re(-\mathrm{i}\xi_{P,3})>0$, and we can then construct $\varphi_P^-$ in a neighborhood of $p$ such that $\partial_{x_3}(\varphi_P^-)=-\xi_{P,3}$. Then $\Re(\mathrm{i}\varphi_P^-)<0$ for $x_3<0$ and consequently
\[
|e^{\mathrm{i}\varrho\varphi_P^-}|=\mathcal{O}(|x_3|^\infty)
\]
near $p$ in $\mathbb{R}\times\Omega$. Then we can construct $u_\varrho^-$ of the form \eqref{reflectedwaves} in a neighborhood of $p$ such that
\[
\rho\partial_t^2 u_{P,\varrho}^--\nabla\cdot S^L(u_{P,\varrho}^-)=0,\quad\text{up to order }N\text{ at point }p,
\]
and now $\chi^-_P$ is compactly supported in a neighborhood of $p$. We also refer to \cite{stefanov2021transmission} for a discussion on evanescent waves.
\subsection{Full Gaussian beam asymptotic solutions}
Now it is clear to see how to construct Gaussian beam solutions to the linear elastic wave equation incorporating all reflections at the boundary $(0,T)\times\partial \Omega$.

Assume $\vartheta:(t^-,t^+)\rightarrow M$ be a unit speed null-geodesic in $(M,-\mathrm{d}t^2+c_\bullet^{-2}\mathrm{d}s^2)$, with endpoints $\vartheta(t^-),\vartheta(t^+)\in\partial M$. Let $R_0\subset\partial M$ be a neighborhood of $\vartheta(t^-)$. Assume that $t^-\in (0,T)$. If $\vartheta$ is a forward null-geodesic then $t^+>t^-$, if $\vartheta$ is a backward null-geodesic then $t^+<t^-$.

Fix $k$ and $K$, by above discussions, we can take $N$ large enough and construct asymptotic solutions $u_\varrho$ such that
\[
\|\mathcal{L}_{\lambda,\mu,\rho}u_\varrho\|_{H^k((0,T)\times\Omega)}=\mathcal{O}(\varrho^{-K}),
\]
the boundary values of $u_\varrho$ satisfy
\[
\|\mathcal{N}_{\lambda,\mu}u_\varrho\|_{H^k((0,T)\times\partial\Omega\setminus R_0)}=\mathcal{O}(\varrho^{-K}).
\]
Actually, $u_\varrho=u_\varrho^{\mathrm{inc}}+u_\varrho^{\mathrm{ref}}$, where the incident wave $u_\varrho^{\mathrm{inc}}=u_\varrho^+$ is compactly supported in a neighborhood of $\vartheta$.
We remark here that $u_\varrho^{\mathrm{inc}}$ is a Gaussian beam starting from $\vartheta(t_-)$, and $u_\rho^{\mathrm{ref}}$ is generated by the reflection of the incident wave $u_\varrho^{\mathrm{inc}}$ at $\vartheta(t_+)$ and subsequent reflections.

\section{Proof of the main theorem}\label{mainproof}
\subsection{Second order linearization of the displacement-to-traction map}\label{SOL}
We first summarize the second order linearization of the displacement-to-traction map carried out in \cite{uhlmann2021inverse}. We also refer to \cite{kurylev2018inverse,lassas2018inverse,hintz2022inverse} for the use of higher order linearization in the study of inverse problems for nonlinear equations.\\

Take $\epsilon_1,\epsilon_2\in\mathbb{R}$ small enough, and let $u_\epsilon$ be the solution to the initial boundary value problem \eqref{elastic_eq} with boundary value $f=\epsilon_1f^{(1)}+\epsilon_2f^{(2)}$, then $u_\epsilon$ has the asymptotic expansion
\[
u_\epsilon=\epsilon_1u^{(1)}+\epsilon_2u^{(2)}+\frac{1}{2}\epsilon_1^2u^{(11)}+\frac{1}{2}\epsilon_2^2u^{(22)}+\epsilon_1\epsilon_2u^{(12)}+\text{higher order terms in }\epsilon_1,\epsilon_2.
\]
Here $u^{(1)},u^{(2)}$ are solutions to the linearized equation
\begin{equation}\label{linear_eqj}
\begin{split}
&\rho\frac{\partial^2u^{(j)}}{\partial t^2}-\nabla\cdot S^L(x,u^{(j)})=0,\quad (t,x)\in (0,T)\times\Omega,\\
&u^{(j)}(t,x)=f^{(j)}(t,x),\quad (t,x)\in (0,T)\times\partial \Omega,\\
&u^{(j)}(0,x)=\frac{\partial}{\partial t}u^{(j)}(0,x)=0,\quad x\in \Omega,
\end{split}
\end{equation}
and $u^{(12)}$ is the solution to the equation
\begin{equation}\label{linear_eq12}
\begin{split}
&\rho\frac{\partial^2u^{(12)}}{\partial t^2}-\nabla\cdot S^L(x,u^{(12)})=\nabla\cdot G(u^{(1)},u^{(2)}),\quad (t,x)\in (0,T)\times\Omega,\\
&u^{(12)}(t,x)=0,\quad (t,x)\in (0,T)\times\partial \Omega,\\
&u^{(12)}(0,x)=\frac{\partial}{\partial t}u^{(12)}(0,x)=0,\quad x\in \Omega,
\end{split}
\end{equation}
where $G_{ij}$ comes from the second order term of $u$ in $S(x,u)$,
\begin{equation}\label{G_form}
\begin{split}
G_{ij}(u^{(1)},u^{(2)})=&(\lambda+\mathscr{B})\frac{\partial u^{(1)}_m}{\partial x_n}\frac{\partial u^{(2)}_m}{\partial x_n}\delta_{ij}+2\mathscr{C}\frac{\partial u^{(1)}_m}{\partial x_m}\frac{\partial u^{(2)}_n}{\partial x_n}\delta_{ij}+\mathscr{B}\frac{\partial u_m^{(1)}}{\partial x_n}\frac{\partial u_n^{(2)}}{\partial x_m}\delta_{ij}\\
&+\mathscr{B}\left(\frac{\partial u^{(1)}_m}{\partial x_m}\frac{\partial u^{(2)}_j}{\partial x_i}+\frac{\partial u^{(2)}_m}{\partial x_m}\frac{\partial u^{(1)}_j}{\partial x_i}\right)+\frac{\mathscr{A}}{4}\left(\frac{\partial u^{(1)}_j}{\partial x_m}\frac{\partial u^{(2)}_m}{\partial x_i}+\frac{\partial u^{(2)}_j}{\partial x_m}\frac{\partial u^{(1)}_m}{\partial x_i}\right)\\
&+(\lambda+\mathscr{B})\left(\frac{\partial u^{(1)}_m}{\partial x_m}\frac{\partial u^{(2)}_i}{\partial x_j}+\frac{\partial u^{(2)}_m}{\partial x_m}\frac{\partial u^{(1)}_i}{\partial x_j}\right)\\
&+\left(\mu+\frac{\mathscr{A}}{4}\right)\Bigg(\frac{\partial u^{(1)}_m}{\partial x_i}\frac{\partial u^{(2)}_m}{\partial x_j}+\frac{\partial u^{(2)}_m}{\partial x_i}\frac{\partial u^{(1)}_m}{\partial x_j}+\frac{\partial u^{(1)}_i}{\partial x_m}\frac{\partial u^{(2)}_j}{\partial x_m}+\frac{\partial u^{(2)}_i}{\partial x_m}\frac{\partial u^{(1)}_j}{\partial x_m}\\
&\quad\quad+\frac{\partial u^{(1)}_i}{\partial x_m}\frac{\partial u^{(2)}_m}{\partial x_j}+\frac{\partial u^{(2)}_i}{\partial x_m}\frac{\partial u^{(1)}_m}{\partial x_j}\Bigg).
\end{split}
\end{equation}

%Notice that now we obtain a second order linearization of the traction-to-displacement map
%\[
%\frac{\partial^2}{\partial\epsilon_1\partial\epsilon_2}\Lambda(\epsilon_1f^{(1)}+\epsilon_2f^{(2)})\vert_{\epsilon_1=\epsilon_2=0}=u^{(12)}\vert_{(0,T)\times\partial\Omega}.
%\]
%We remark here that a first order linearization leads to
%\begin{equation}\label{firstlin}
%\frac{\partial}{\partial\epsilon_j}\Lambda(\epsilon_jf^{(j)})\vert_{\epsilon_j=0}=u^{(j)}\vert_{(0,T)\times\partial\Omega}=\Lambda^{\mathrm{lin}}(f^{(j)}).
%\end{equation}
We define the linear map
\[
\begin{split}
\Lambda^{'}_{D}:f^{(j)}&\mapsto \nu\cdot S^L(x,u^{(j)})\vert_{(0,T)\times\partial\Omega},\\
\end{split}
\]
and the bilinear map
\[
\begin{split}
\Lambda^{''}_{D}:(f^{(1)},f^{(2)})&\mapsto \left(\nu\cdot S^L(u^{(12)})+\nu\cdot G(u^{(1)},u^{(2)})\right)\Big\vert_{(0,T)\times\partial\Omega}.\\
\end{split}
\]
Here $f^{(1)},f^{(2)}$ vanish near $\{t=0\}$. We remark here that formally 
\[
\begin{split}
\Lambda^{'}_{D}(f^{(j)})&=\frac{\partial}{\partial\epsilon_j}\Lambda(\epsilon_jf^{(j)})\vert_{\epsilon_j=0},\\
\Lambda^{''}_{D}(f^{(1)},f^{(2)})&=\frac{\partial^2}{\partial\epsilon_1\partial\epsilon_2}\Lambda(\epsilon_1f^{(1)}+\epsilon_2f^{(2)})\vert_{\epsilon_1=\epsilon_2=0}.
\end{split}
\]
Therefore, we can recover $\Lambda'_D,\Lambda''_D$ from $\Lambda$.

Assume that $u^{(0)}$ solves the initial boundary value problem for the backward elastic wave equation
\begin{equation}\label{backward_eq}
\begin{split}
&\rho\frac{\partial^2}{\partial t^2}u^{(0)}-\nabla\cdot S^L(x,u^{(0)})=0,\quad (t,x)\in (0,T)\times\Omega,\\
&u(t,x)=f^{(0)}(t,x),\quad (t,x)\in (0,T)\times\partial \Omega,\\
&u^{(0)}(T,x)=\frac{\partial}{\partial t}u^{(0)}(T,x)=0,\quad x\in \Omega.
\end{split}
\end{equation}
Using integration by parts, we have (cf. \cite{uhlmann2021inverse})
\begin{equation}\label{integralform}
\int_0^T\int_{\partial \Omega}\Lambda^{''}_{D}(f^{(1)},f^{(2)})\cdot f^{(0)}\mathrm{d}S\mathrm{d}t\\
=\int_0^T\int_{\Omega}\mathcal{G}(\nabla u^{(1)},\nabla u^{(2)},\nabla u^{(0)})\mathrm{d}x\mathrm{d}t,
\end{equation}
where
\begin{equation}\label{integrand_G}
\begin{split}
&\mathcal{G}(\nabla u^{(1)},\nabla u^{(2)},\nabla u^{(0)})\\
=&(\lambda+\mathscr{B})(\nabla u^{(1)}:\nabla u^{(2)})(\nabla\cdot u^{(0)})+2\mathscr{C} (\nabla\cdot u^{(1)})(\nabla\cdot u^{(2)})(\nabla\cdot u^{(0)})\\
&+\mathscr{B}(\nabla u^{(1)}:\nabla^T u^{(2)})(\nabla\cdot u^{(0)})\\
&+\mathscr{B}\left( (\nabla\cdot u^{(1)})(\nabla u^{(2)}:\nabla^T u^{(0)})+(\nabla\cdot u^{(2)})(\nabla u^{(1)}:\nabla^T u^{(0)})\right)\\
&+\frac{\mathscr{A}}{4}\left(\frac{\partial u^{(1)}_j}{\partial x_m}\frac{\partial u^{(2)}_m}{\partial x_i}+\frac{\partial u^{(2)}_j}{\partial x_m}\frac{\partial u^{(1)}_m}{\partial x_i}\right)\frac{\partial u^{(0)}_i}{\partial x_j}\\
&+(\lambda+\mathscr{B})\left((\nabla\cdot u^{(1)})(\nabla u^{(2)}:\nabla u^{(0)})+(\nabla\cdot u^{(2)})(\nabla u^{(1)}:\nabla u^{(0)})\right)\\
&+\left(\mu+\frac{\mathscr{A}}{4}\right)\Bigg(\frac{\partial u^{(1)}_m}{\partial x_i}\frac{\partial u^{(2)}_m}{\partial x_j}+\frac{\partial u^{(2)}_m}{\partial x_i}\frac{\partial u^{(1)}_m}{\partial x_j}+\frac{\partial u^{(1)}_i}{\partial x_m}\frac{\partial u^{(2)}_j}{\partial x_m}+\frac{\partial u^{(2)}_i}{\partial x_m}\frac{\partial u^{(1)}_j}{\partial x_m}\\
&\quad\quad\quad\quad\quad\quad\quad\quad\quad+\frac{\partial u^{(1)}_i}{\partial x_m}\frac{\partial u^{(2)}_m}{\partial x_j}+\frac{\partial u^{(2)}_i}{\partial x_m}\frac{\partial u^{(1)}_m}{\partial x_j}\Bigg)\frac{\partial u^{(0)}_i}{\partial x_j}.
\end{split}
\end{equation}

\noindent\textbf{Linearized traction-to-displacement map.} 
Let $v$ be the solution to the initial boundary value problem with Neumann boundary value
\begin{equation}\label{linear_eqj_v}
\begin{split}
&\rho\frac{\partial^2v}{\partial t^2}-\nabla\cdot S^L(x,v)=0,\quad (t,x)\in (0,T)\times\Omega,\\
&\nu\cdot S^L(x,v)=g,\quad (t,x)\in (0,T)\times\partial \Omega,\\
&v(0,x)=\frac{\partial}{\partial t}v(0,x)=0,\quad x\in \Omega,
\end{split}
\end{equation}
where $g$ is supported away from $\{t=0\}$.
Denote the Neumann-to-Dirichlet map
\[
\Lambda^{'}_{N}:g\mapsto v\vert_{(0,T)\times\partial\Omega}.
\]
It is clear to see that
\[
\Lambda^{'}_{N}=(\Lambda^{'}_{D})^{-1}.
\]

Let $v^{(j)}, j=1,2$ be solution to \eqref{linear_eqj_v} with $g=g^{(1)}$, and $v^{(12)}$ be solution to the following initial boundary value problem
\begin{equation}\label{linear_eq12_v}
\begin{split}
&\rho\frac{\partial^2v^{(12)}}{\partial t^2}-\nabla\cdot S^L(x,u^{(12)})=\nabla\cdot G(v^{(1)},v^{(2)}),\quad (t,x)\in (0,T)\times\Omega,\\
&\nu\cdot S^L(x,v^{(12)})=-\nu\cdot G(v^{(1)},v^{(2)}),\quad (t,x)\in (0,T)\times\partial \Omega,\\
&v^{(12)}(0,x)=\frac{\partial}{\partial t}v^{(12)}(0,x)=0,\quad x\in \Omega.
\end{split}
\end{equation}
Denote then
\[
\Lambda^{''}_{N}:(g^{(1)},g^{(2)})\mapsto v^{(12)}\vert_{(0,T)\times\partial\Omega}.
\]

Assume $v^{(0)}$ solves the initial boundary value problem for the backward elastic wave equation
\begin{equation}\label{backward_eq}
\begin{split}
&\rho\frac{\partial^2}{\partial t^2}v^{(0)}-\nabla\cdot S^L(x,v^{(0)})=0,\quad (t,x)\in (0,T)\times\Omega,\\
&\nu\cdot S^L(x,v^{(0)})(t,x)=g^{(0)},\quad (t,x)\in (0,T)\times\partial \Omega,\\
&v^{(0)}(T,x)=\frac{\partial}{\partial t}v^{(0)}(T,x)=0,\quad x\in \Omega.
\end{split}
\end{equation}
Using integration by parts, we calculate
\begin{equation}\label{integralform}
\begin{split}
&\int_0^T\int_{\partial \Omega}\Lambda^{''}_{N}(g^{(1)},g^{(2)})\cdot g^{(0)}\mathrm{d}S\mathrm{d}t\\
=&\int_0^T\int_{\partial \Omega}v^{(12)}\cdot (\nu\cdot S^L(x,v^{(0)}))\mathrm{d}S\mathrm{d}t\\
=&\int_0^T\int_{\Omega}(\nabla\cdot S^L(x,v^{(0)}))\cdot v^{(12)}\mathrm{d}x\mathrm{d}t+\int_0^T\int_\Omega S^L(x,v^{(0)}):\nabla v^{(12)}\mathrm{d}x\mathrm{d}t\\
=&\int_0^T\int_{\Omega}\rho\frac{\partial^2v^{(0)}}{\partial t^2}\cdot v^{(12)}\mathrm{d}x\mathrm{d}t-\int_0^T\int_\Omega v^{(0)}\cdot (\nabla\cdot S^L(x,v^{(12)}))\mathrm{d}x\mathrm{d}t\\
&\quad\quad\quad\quad\quad\quad\quad\quad\quad\quad\quad\quad\quad\quad\quad\quad+\int_0^T\int_{\partial \Omega}v^{(0)}\cdot(\nu\cdot S^L(x,v^{(12)})\mathrm{d}S\mathrm{d}t\\
=&\int_0^T\int_{\Omega}v^{(0)}\cdot\left(\rho\frac{\partial^2v^{(12)}}{\partial t^2}-\nabla\cdot S^L(x,v^{(12)})\right)\mathrm{d}x\mathrm{d}t-\int_0^T\int_{\partial \Omega}v^{(0)}\cdot(\nu\cdot G(v^{(1)},v^{(2)}))\mathrm{d}S\mathrm{d}t\\
=&\int_0^T\int_{\Omega}v^{(0)}\cdot(\nabla\cdot G(v^{(1)},v^{(2)}))\mathrm{d}x\mathrm{d}t-\int_0^T\int_{\partial \Omega}v^{(0)}\cdot(\nu\cdot G(v^{(1)},v^{(2)}))\mathrm{d}S\mathrm{d}t\\
=&-\int_0^T\int_{\Omega}\mathcal{G}(\nabla v^{(1)},\nabla v^{(2)},\nabla v^{(0)})\mathrm{d}x\mathrm{d}t,
\end{split}
\end{equation}
Note that
\[
\begin{split}
&\int_0^T\int_{\partial \Omega}\Lambda^{''}_{N}(g^{(1)},g^{(2)})\cdot g^{(0)}\mathrm{d}S\mathrm{d}t\\
=&-\int_0^T\int_{\Omega}\mathcal{G}(\nabla v^{(1)},\nabla v^{(2)},\nabla v^{(0)})\mathrm{d}x\mathrm{d}t\\
=&-\int_0^T\int_{\partial \Omega}\Lambda^{''}_{D}(v^{(1)},v^{(2)})\cdot v^{(0)}\mathrm{d}S\mathrm{d}t\\
=&-\int_0^T\int_{\partial \Omega}\Lambda^{''}_{D}(\Lambda^{'}_{N}(g^{(1)}),\Lambda^{'}_{N}(g^{(2)}))\cdot\Lambda^{'}_{N}(g^{(0)})\mathrm{d}S\mathrm{d}t.
\end{split}
\]

Therefore, for the rest of the paper, we only need to consider the problem of recovering the parameters $\lambda,\mu,\rho,\mathscr{A},\mathscr{B},\mathscr{C}$ from $\Lambda^{'}_{D}$ and $\Lambda^{''}_{N}$.\\

Now assume that $\Lambda$ is the displacement-to-traction map associated with $\lambda,\mu,\rho,\mathscr{A},\mathscr{B},\mathscr{C}$, and $\widetilde{\Lambda}$ is the displacement-to-traction map associated with $\widetilde{\lambda},\widetilde{\mu},\widetilde{\rho},\widetilde{\mathscr{A}},\widetilde{\mathscr{B}},\widetilde{\mathscr{C}}$. Our goal is to prove that $\Lambda=\widetilde{\Lambda}$ implies
\[
(\lambda,\mu,\rho,\mathscr{A},\mathscr{B},\mathscr{C})=(\widetilde{\lambda},\widetilde{\mu},\widetilde{\rho},\widetilde{\mathscr{A}},\widetilde{\mathscr{B}},\widetilde{\mathscr{C}}).
\]
\subsection{Uniqueness of the wavespeeds}
Now assume $\Lambda'_D=\widetilde{\Lambda}'_D$ (or equivalently $\Lambda^{\mathrm{lin}}=\widetilde{\Lambda}^{\mathrm{lin}}$) and denote
\[
g_P=c_P^{-2}\mathrm{d}s^2=\frac{\rho}{\lambda+2\mu}\mathrm{d}s^2,\quad g_S=c_S^{-2}\mathrm{d}s^2=\frac{\rho}{\mu}\mathrm{d}s^2,
\]
and similarly
\[
\widetilde{g}_P=\widetilde{c}_P^{-2}\mathrm{d}s^2=\frac{\widetilde{\rho}}{\widetilde{\lambda}+2\widetilde{\mu}}\mathrm{d}s^2,\quad \widetilde{g}_S=\widetilde{c}_S^{-2}\mathrm{d}s^2=\frac{\widetilde{\rho}}{\widetilde{\mu}}\mathrm{d}s^2.
\]

%We first get  by first order linearization of $\Lambda=\widetilde{\Lambda}$ (cf. \eqref{firstlin})
%\[
%\frac{\partial}{\partial\epsilon}\Lambda(\epsilon f)\vert_{\epsilon=0}=\Lambda^{\mathrm{lin}}(f).
%\]
For the uniqueness of the two wavespeeds, we summarize the results in \cite{rachele2000inverse}, \cite{stefanov2017local} in the following proposition.
\begin{proposition}
Assume that $(\Omega,g_\bullet)$ and $(\Omega,\widetilde{g}_\bullet)$, where $\bullet=P,S$, satisfy either of the following two conditions
\begin{enumerate}
\item $(\Omega,g_\bullet/\widetilde{g}_\bullet)$ is simple;
\item $(\Omega,g_\bullet/\widetilde{g}_\bullet)$ satisfies the strictly convex foliation condition.
\end{enumerate}
 Then $\Lambda^{\mathrm{lin}}=\widetilde{\Lambda}^{\mathrm{lin}}$ implies that
\begin{equation}\label{equal_lmr}
c_P=\widetilde{c}_P,\quad c_S=\widetilde{c}_S, \text{ or equivalently,  }\quad\frac{\mu}{\rho}=\frac{\widetilde{\mu}}{\widetilde{\rho}},\quad\frac{\lambda}{\rho}=\frac{\widetilde{\lambda}}{\widetilde{\rho}},\quad\text{in }\overline{\Omega}.
\end{equation}
\end{proposition}
From now on we assume the equality \eqref{equal_lmr} to hold, and therefore $g_\bullet=\widetilde{g}_\bullet$ for $\bullet=S,P$. We only need to work with the metrics $g_\bullet=c_\bullet^{-2}\mathrm{d}s^2$ in the following.

\subsection{An integral identity}\label{threegaussians}
 Fix $k$ large enough. By previous section, we can construct $u_\varrho$ and $\widetilde{u}_\varrho$ such that
 \begin{equation}\label{est_rmder}
 \|\mathcal{L}_{\lambda,\mu,\rho}u_\varrho\|_{H^k}=\mathcal{O}(\varrho^{-K}),\quad \| \mathcal{L}_{\widetilde{\lambda},\widetilde{\mu},\widetilde{\rho}}\widetilde{u}_\varrho\|_{H^k}=\mathcal{O}(\varrho^{-K}),
 \end{equation}
 where $u_\varrho=u_\varrho^{\mathrm{inc}}+u_\varrho^{\mathrm{ref}}$, $\widetilde{u}_\varrho=\widetilde{u}_\varrho^{\mathrm{inc}}+\widetilde{u}_\varrho^{\mathrm{ref}}$, and $u^{\mathrm{inc}}_\varrho$ and $\widetilde{u}^{\mathrm{inc}}_\varrho$ are Gaussian beam solutions concentrating near a null-geodesic $\vartheta$ in $(M,-\mathrm{d}t^2+c_\bullet^{-2}\mathrm{d}s^2)$. Furthermore we have
 \[
\|\mathcal{N}_{\lambda,\mu}u_\varrho\|_{H^k((0,T)\times\partial\Omega\setminus R_0)}=\mathcal{O}(\varrho^{-K}),\quad \|\mathcal{N}_{\widetilde{\lambda},\widetilde{\mu}}\widetilde{u}_\varrho\|_{H^k((0,T)\times\partial\Omega\setminus R_0)}=\mathcal{O}(\varrho^{-K}).
\]

Since $\Lambda^{\mathrm{lin}}=\widetilde{\Lambda}^{\mathrm{lin}}$, we know that jets of $(\lambda,\mu,\rho)$ and $(\widetilde{\lambda},\widetilde{\mu},\widetilde{\rho})$ at $\partial\Omega$ are equal (cf. \cite{rachele2000boundary}). Therefore we can denote
\[
\mathcal{N}=\mathcal{N}_{\lambda,\mu}=\mathcal{N}_{\widetilde{\lambda},\widetilde{\mu}}.
\]
Also we can extend $(\lambda,\mu,\rho)$ and $(\widetilde{\lambda},\widetilde{\mu},\widetilde{\rho})$ smoothly to $\widetilde{\Omega}$ such that $(\lambda,\mu,\rho)=(\widetilde{\lambda},\widetilde{\mu},\widetilde{\rho})$ on $\widetilde{\Omega}\setminus\Omega$. By the constructions in Section \ref{gaussianbeams}, we can arrange the boundary values of $u_\varrho^{\mathrm{inc}}$ and $\widetilde{u}_\varrho^{\mathrm{inc}}$ such that
\begin{equation}\label{diff_bdryvalue}
\|\mathcal{N}(u_\varrho-\widetilde{u}_\varrho)\|_{H^k(R_0)}=\|\mathcal{N}(u^{\mathrm{inc}}_\varrho-\widetilde{u}^{\mathrm{inc}}_\varrho)\|_{H^k(R_0)}\leq C\varrho^{-K}.
\end{equation}
We refer to \cite{hintz2022dirichlet} for more details.

Let $f_\varrho=u_\varrho\vert_{(0,T)\times\partial\Omega}$.
Then one can construct solutions $u$ and $\widetilde{u}$ to
\begin{equation}\label{eq_u}
\begin{split}
\mathcal{L}_{\lambda,\mu,\rho}u&=0,\quad \text{on }(0,T)\times\Omega,\\
\mathcal{N}u&=f_{\varrho},\quad\text{on }(0,T)\times\partial\Omega,\\
u(0,x)=\partial_tu(0,x)&=0,\quad\text{for }x\in\Omega.
\end{split}
\end{equation}
and
\begin{equation}\label{eq_utilde}
\begin{split}
\mathcal{L}_{\widetilde{\lambda},\widetilde{\mu},\widetilde{\rho}}\widetilde{u}&=0,\quad \text{on }(0,T)\times\Omega,\\
\mathcal{N}\widetilde{u}&=f_{\varrho},\quad\text{on }(0,T)\times\partial\Omega,\\
\widetilde{u}(0,x)=\partial_t\widetilde{u}(0,x)&=0,\quad\text{for }x\in\Omega.
\end{split}
\end{equation}
To be more precise, we construct above solutions such that
\[
u=u_\varrho+R_\varrho,\quad \widetilde{u}=\widetilde{u}_\varrho+\widetilde{R}_\varrho,
\]
where $R_\varrho$ and $\widetilde{R}_\varrho$ satisfy
\begin{equation}\label{eq_R}
\begin{split}
\mathcal{L}_{\lambda,\mu,\rho}R_\varrho&=-\mathcal{L}_{\lambda,\mu,\rho}u_\varrho,\quad \text{on }(0,T)\times\Omega,\\
\mathcal{N}R&=0,\quad\text{on }(0,T)\times\partial\Omega,\\
R(0,x)=\partial_tR(0,x)&=0,\quad\text{for }x\in\Omega,
\end{split}
\end{equation}
and
\begin{equation}\label{eq_Rtilde}
\begin{split}
\mathcal{L}_{\widetilde{\lambda},\widetilde{\mu},\widetilde{\rho}}\widetilde{R}_\varrho&=-\mathcal{L}_{\widetilde{\lambda},\widetilde{\mu},\widetilde{\rho}}\widetilde{u}_\varrho,\quad \text{on }(0,T)\times\Omega,\\
\mathcal{N}\widetilde{R}&=(u_\varrho-\widetilde{u}_\varrho)\vert_{(0,T)\times\partial\Omega},\quad\text{on }(0,T)\times\partial\Omega,\\
\widetilde{R}(0,x)=\partial_t\widetilde{R}(0,x)&=0,\quad\text{for }x\in\Omega.
\end{split}
\end{equation}
By \eqref{est_rmder} and \eqref{diff_bdryvalue}, one can obtain 
\[
\|R_\varrho\|_{H^{k+1}},\|\widetilde{R}_\varrho\|_{H^{k+1}}\leq C\varrho^{-K}.
\]
using standard theory for linear hyperbolic systems. Take $K$ large enough, then if $k>2$ we can use Sobolev imbedding to have
\[
\|R_\varrho\|_{W^{1,3}},\|\widetilde{R}_\varrho\|_{W^{1,3}}=\mathcal{O}(\varrho^{-1/2}).
\]
We note also that
\[
\|u_\varrho\|_{W^{1,3}},\|\widetilde{u}_\varrho\|_{W^{1,3}}=\mathcal{O}(\varrho^{1/2}).
\]
\\

Let $\vartheta^{(j)}$ be a null-geodesic, and construct Gaussian beam solutions $u^{(j)}_{\kappa_j\varrho}$ and $\widetilde{u}^{(j)}_{\kappa_j\varrho}$, whose incident waves $u^{(j),\mathrm{inc}}_{\kappa_j\varrho}$ and $\widetilde{u}^{(j),\mathrm{inc}}_{\kappa_j\varrho}$ concentrate near $\vartheta^{(j)}$, where $\kappa_j$ is a constant. Then we construct $u^{(j)}$ and $\widetilde{u}^{(j)}$, $j=1,2$ be solutions to \eqref{eq_u} and \eqref{eq_utilde} with $f_{\kappa_j\varrho}=u^{(j)}_{\kappa_j\varrho}\vert_{(0,T)\times\partial\Omega}$. Also, let $\vartheta^{(0)}$ be a backward null-geodesic, and construct Gaussian beam solution $u^{(0)}_{\kappa_0\varrho}$, $\widetilde{u}^{(0)}_{\kappa_0\varrho}$, whose incident waves $u^{(0),\mathrm{inc}}_{\kappa_0\varrho}$ and $\widetilde{u}^{(0),\mathrm{inc}}_{\kappa_0\varrho}$ concentrate near $\vartheta^{(0)}$. Then let $f^{(0)}_{\kappa_0\varrho}=u^{(0)}_{\kappa_0\varrho}\vert_{(0,T)\times\partial\Omega}$ and construct exact solutions $u^{(0)}$, $\widetilde{u}^{(0)}$ to the backward elastic wave equations
\begin{equation}\label{eq_u0}
\begin{split}
\mathcal{L}_{\lambda,\mu,\rho}u^{(0)}&=0,\quad \text{on }(0,T)\times\Omega,\\
\mathcal{N}u^{(0)}&=f^{(0)}_{\kappa_0\varrho},\quad\text{on }(0,T)\times\partial\Omega,\\
u^{(0)}(T,x)=\partial_tu^{(0)}(T,x)&=0,\quad\text{for }x\in\Omega,
\end{split}
\end{equation}
and
\begin{equation}\label{eq_u0tilde}
\begin{split}
\mathcal{L}_{\widetilde{\lambda},\widetilde{\mu},\widetilde{\rho}}\widetilde{u}^{(0)}&=0,\quad \text{on }(0,T)\times\Omega,\\
\mathcal{N}\widetilde{u}^{(0)}&=f^{(0)}_{\kappa_0\varrho},\quad\text{on }(0,T)\times\partial\Omega,\\
\widetilde{u}^{(0)}(T,x)=\partial_t\widetilde{u}^{(0)}(T,x)&=0,\quad\text{for }x\in\Omega.
\end{split}
\end{equation}

Then we carry out the second order linearization of traction-to-displacement map $\Lambda$ (and $\widetilde{\Lambda}$), the identity
\[
\Lambda=\widetilde{\Lambda},
\]
yields
\[
\int_0^T\int_{\partial\Omega}\Lambda_N''(f^{(1)}_{\kappa_1\varrho},f^{(2)}_{\kappa_2\varrho})\cdot f^{(0)}_{\kappa_0\varrho}\mathrm{d}S\mathrm{d}t=\int_0^T\int_{\partial\Omega}\widetilde{\Lambda}_N''(f^{(1)}_{\kappa_1\varrho},f^{(2)}_{\kappa_2\varrho})\cdot f^{(0)}_{\kappa_0\varrho}\mathrm{d}S\mathrm{d}t,
\]
and using \eqref{integralform} we obtain
\begin{equation}\label{eqaulityI}
\mathcal{I}:=\int_0^T\int_{\Omega}\mathcal{G}(\nabla u^{(1)},\nabla u^{(2)},\nabla u^{(0)})\,\mathrm{d}x\mathrm{d}t=\int_0^T\int_{\Omega}\widetilde{\mathcal{G}}(\nabla \widetilde{u}^{(1)},\nabla \widetilde{u}^{(2)},\nabla \widetilde{u}^{(0)})\,\mathrm{d}x\mathrm{d}t=:\widetilde{\mathcal{I}},
\end{equation}
In the following, for the sake of simplicity we denote
\[
\begin{split}
u_{\kappa_1\varrho}^{(1),\mathrm{inc}}=e^{\mathrm{i}\kappa_1\varrho\varphi^{(1)}}\chi^{(1)}(\mathbf{a}^{(1)}+\mathcal{O}(\varrho^{-1})),\\
u_{\kappa_2\varrho}^{(2),\mathrm{inc}}=e^{\mathrm{i}\kappa_2\varrho\varphi^{(2)}}\chi^{(2)}(\mathbf{a}^{(2)}+\mathcal{O}(\varrho^{-1})),\\
u_{\kappa_0\varrho}^{(0),\mathrm{inc}}=e^{\mathrm{i}\kappa_0\varrho\varphi^{(0)}}\chi^{(0)}(\mathbf{a}^{(0)}+\mathcal{O}(\varrho^{-1})),
\end{split}
\]
where $\mathbf{a}^{(j)}:=\mathbf{a}_0^{(j)}$.
If the support of $u_{\kappa_1\varrho}^{(1)}u_{\kappa_2\varrho}^{(2)}u_{\kappa_0\varrho}^{(0)}$ is equal to the support of $u_{\kappa_1\varrho}^{(1),\mathrm{inc}}u_{\kappa_2\varrho}^{(2),\mathrm{inc}}u_{\kappa_0\varrho}^{(0),\mathrm{inc}}$ (which is the case in the following proof), we have (see \cite{uhlmann2021inverse} for more details)
\begin{equation}\label{Iasym}
\varrho^{-1}\frac{\mathrm{i}}{\kappa_1\kappa_2\kappa_3}\mathcal{I}=\rho^2\int_0^T\int_\Omega e^{\mathrm{i}\varrho S}\chi^{(1)}\chi^{(2)}\chi^{(0)}\mathcal{A}\mathrm{d}x\mathrm{d}t+\mathcal{O}(\varrho^{-1/2}),
\end{equation}
where
\begin{equation}\label{def_S}
S=\kappa_1\varphi^{(1)}+\kappa_2\varphi^{(2)}+\kappa_0\varphi^{(0)}
\end{equation}
and
\[
\mathcal{A}:=\mathcal{G}(\mathbf{a}^{(1)}\otimes\nabla\varphi^{(1)},\mathbf{a}^{(2)}\otimes\nabla\varphi^{(2)},\mathbf{a}^{(0)}\otimes\nabla\varphi^{(0)})
\]
with
\[
\begin{split}
&\mathcal{G}(\mathbf{a}^{(1)}\otimes\nabla\varphi^{(1)},\mathbf{a}^{(2)}\otimes\nabla\varphi^{(2)},\mathbf{a}^{(0)}\otimes\nabla\varphi^{(0)})\\
=&\mathscr{B}[(\mathbf{a}^{(1)}\cdot\nabla\varphi^{(1)})(\mathbf{a}^{(2)}\cdot\nabla\varphi^{(0)})(\mathbf{a}^{(0)}\cdot\nabla\varphi^{(2)})+(\mathbf{a}^{(2)}\cdot\nabla\varphi^{(2)})(\mathbf{a}^{(1)}\cdot\nabla\varphi^{(0)})(\mathbf{a}^{(0)}\cdot\nabla\varphi^{(1)})\\
&\quad\quad\quad\quad+ (\mathbf{a}^{(2)}\cdot\nabla\varphi^{(1)})(\mathbf{a}^{(1)}\cdot\nabla\varphi^{(2)})(\mathbf{a}^{(0)}\cdot\nabla\varphi^{(0)})]\\
&+\frac{\mathscr{A}}{4}\left((\mathbf{a}^{(2)}\cdot\nabla\varphi^{(1)})(\mathbf{a}^{(1)}\cdot\nabla\varphi^{(0)})(\mathbf{a}^{(0)}\cdot\nabla\varphi^{(2)})+(\mathbf{a}^{(1)}\cdot\nabla\varphi^{(2)})(\mathbf{a}^{(2)}\cdot\nabla\varphi^{(0)})(\mathbf{a}^{(0)}\cdot\nabla\varphi^{(1)})\right)\\
&+(\lambda+\mathscr{B})[(\mathbf{a}^{(1)}\cdot\nabla\varphi^{(1)})(\mathbf{a}^{(2)}\cdot\mathbf{a}^{(0)})(\nabla\varphi^{(2)}\cdot\nabla\varphi^{(0)})+(\mathbf{a}^{(2)}\cdot\nabla\varphi^{(2)})(\mathbf{a}^{(1)}\cdot\mathbf{a}^{(0)})(\nabla\varphi^{(1)}\cdot\nabla\varphi^{(0)})\\
&\quad\quad+(\mathbf{a}^{(1)}\cdot\mathbf{a}^{(2)})(\nabla\varphi^{(1)}\cdot\nabla\varphi^{(2)})(\mathbf{a}^{(0)}\cdot\nabla\varphi^{(0)})]+2\mathscr{C}(\mathbf{a}^{(1)}\cdot\nabla\varphi^{(1)})(\mathbf{a}^{(2)}\cdot\nabla\varphi^{(2)})(\mathbf{a}^{(0)}\cdot\nabla\varphi^{(0)})\\
&+(\mu+\frac{\mathscr{A}}{4})\Big((\mathbf{a}^{(1)}\cdot\mathbf{a}^{(2)})(\nabla\varphi^{(1)}\cdot\mathbf{a}^{(0)})(\nabla\varphi^{(2)}\cdot\nabla\varphi^{(0)})+(\nabla\varphi^{(1)}\cdot\nabla\varphi^{(2)})(\mathbf{a}^{(1)}\cdot\mathbf{a}^{(0)})(\mathbf{a}^{(2)}\cdot\nabla\varphi^{(0)})\\
&\quad \quad+(\mathbf{a}^{(2)}\cdot\mathbf{a}^{(1)})(\nabla\varphi^{(2)}\cdot\mathbf{a}^{(0)})(\nabla\varphi^{(1)}\cdot\nabla\varphi^{(0)})+(\mathbf{a}^{(2)}\cdot\mathbf{a}^{(0)})(\mathbf{a}^{(1)}\cdot\nabla\varphi^{(0)})(\nabla\varphi^{(1)}\cdot\nabla\varphi^{(2)})\\
&\quad\quad+(\nabla\varphi^{(1)}\cdot\mathbf{a}^{(2)})(\nabla\varphi^{(2)}\cdot\nabla\varphi^{(0)})(\mathbf{a}^{(1)}\cdot\mathbf{a}^{(0)})+(\nabla\varphi^{(2)}\cdot\mathbf{a}^{(1)})(\nabla\varphi^{(1)}\cdot\nabla\varphi^{(0)})(\mathbf{a}^{(2)}\cdot\mathbf{a}^{(0)})\Big).
\end{split}
\]

%Now, we make a crucial assumption that $\vartheta^{(1)},\vartheta^{(2)}, \vartheta^{(0)}$ intersect at the point $q$ and the intersection of the supports of $u^{(1)},u^{(2)},u^{(0)}$ is a neighborhood of $q$. 

\subsection{Construction of S-S-P waves}
We first introduce the notation
\[
L_q^{\bullet,\pm}M:=\{(\tau,\xi)\in T^*_qM,\,\tau^2=c_\bullet^2|\xi|^2,\,\pm \tau>0\},
\]
where $\bullet=P,S$.

Fix a point $x_0\in\Omega$ and take $q=(\frac{T}{2},x_0)\in M$.
Take $\vartheta^{(j)}:(t_-^{(j)},t_+^{(j)})\rightarrow\Omega$ be a null-geodesic in $((0,T)\times\Omega,\overline{g}_S)$ for $j=1,2$, and $\vartheta^{(0)}:(t_+^{(0)},t_-^{(0)})\rightarrow(0,T)\times\Omega$ be a backward null-geodesic in $((0,T)\times\Omega,\overline{g}_P)$, satisfying the following conditions:
\begin{enumerate}
\item $\vartheta^{(1)},\vartheta^{(2)},\vartheta^{(0)}$ intersect at $q$, that is
\[
\vartheta^{(1)}\left(\frac{T}{2}\right)=\vartheta^{(2)}\left(\frac{T}{2}\right)=\vartheta^{(0)}\left(\frac{T}{2}\right)=p;
\]
\item denote $\zeta^{(j)}=\mathrm{d}\vartheta^{(j)}\vert_q$, then $\zeta^{(1)},\zeta^{(2)}\in L^{S,+}_qM$, $\zeta^{(0)}\in L^{P,-}_qM$ such that
\[
\kappa_1\zeta^{(1)}+\kappa_2\zeta^{(2)}+\kappa_0\zeta^{(0)}=0.
\]
with $\kappa_1,\kappa_2,\kappa_0$ not equal to zero at the same time, $\zeta^{(1)}$ and $\zeta^{(2)}$ are linearly independent.
\end{enumerate}

Since $(\Omega,g_\bullet)$ is non-trapping and $\partial\Omega$ is convex with respect to $g_\bullet$, we can choose $\gamma^{(1)}$ and $\gamma^{(0)}$ such that 
\begin{itemize}
\item $\gamma^{(1)}([t_-^{(1)},\frac{T}{2}])$, as a geodesic in $(\overline{\Omega},g_S)$, has no conjugate points on it;
\item $\gamma^{(0)}([\frac{T}{2},t_-^{(0)}])$, as a geodesic in $(\overline{\Omega},g_P)$, has no conjugate points on it.
\end{itemize}
For $\zeta^{(1)},\zeta^{(0)}$ given, the vector $\zeta^{(2)}$ can be chosen in the following way (see also \cite{de2018nonlinear}). Take $\zeta^{(1)}\in L_q^{S,+}M$, $\zeta^{(0)}\in L_q^{P,-}M$. We write $\zeta^{(1)}=(\tau^{(1)},\xi^{(1)})$, $\zeta^{(0)}=(\tau^{(0)},\xi^{(0)})$, $\xi^{(1)},\xi^{(0)}\in\mathbb{R}^3$, $|\xi^{(1)}|=|\xi^{(0)}|=1$. Then we have
\[
(\tau^{(1)})^2=c_S^2|\xi^{(1)}|^2,\quad (\tau^{(0)})^2=c_P^2|\xi^{(0)}|^2.
\]
Then, we have $\tau^{(1)}=c_S$ and $\tau^{(0)}=-c_P$.
Then we consider the vector $\zeta^{(2)}$ of the form $\zeta^{(2)}=a\zeta^{(1)}+b\zeta^{(0)}$, $a,b\in\mathbb{R}$. Without loss of generality, we take $b=1$. In order for $\zeta^{(2)}$ to be in $L_q^{S,+}M$ and $\zeta^{(1)},\zeta^{(2)}$ are linearly independent, we need
\[
(a\tau^{(1)}+\tau^{(0)})^2=c_S^2|a\xi^{(1)}+\xi^{(0)}|^2.
\]
The above equation, by simple calculation, is equivalent to
\[
2a\left(c_S^2\xi^{(1)}\cdot\xi^{(0)}+c_Sc_P\right)=c_P^2-c_S^2.
\]
Since $|c_S^2\xi^{(1)}\cdot\xi^{(0)}|\leq c_S^2<c_Sc_P$, the above equation always has a solution 
\[
a=\frac{c_P^2-c_S^2}{2(c_S^2\cos\psi+c_Sc_P)}\neq 0.
\]
 Then we get a vector $\zeta^{(2)}\in L_p^{S,+}M$.\\
% Assume  . Then one can extend $\vartheta^{(0)}$ backward to be a broken null-geodesic such that $\vartheta^{(0)}(t):(0,t_+^{(0)})\rightarrow(0,T)\times\Omega$ is a piecewise null-geodesic connecting two boundary points, where each piece could be a null-geodesic in $((0,T)\times\Omega,\overline{g}_S)$ or $((0,T)\times\Omega,\overline{g}_P)$. We also extend $\vartheta^{(j)}$, $j=1,2$, forward in time the same way.

Assume $\vartheta$ is a null-geodesic with endpoints on $\partial M$, we extend  $\vartheta$ to a (collection of) broken null-geodesics $\Upsilon$ in the following inductive way. If $\vartheta([t_0,t_1])\subset\Upsilon$ is a null-geodesic in $(\overline{M},\overline{g}_\bullet)$ joining two boundary boundary points $\vartheta(t_0)$ and $\vartheta(t_1)$ with $t_1<T$, then add the segments $\vartheta^{\mathrm{ref},P}$ and $\vartheta^{\mathrm{ref},S}$ to $\Upsilon$, where 
\begin{itemize}
\item $\vartheta^{\mathrm{ref},P}:[t_1,t_2^P]$ is a null-geodesic in $(\overline{M},\overline{g}_P)$ connecting two boundary points $\vartheta^{\mathrm{ref},P}(t_1)$ and $\vartheta^{\mathrm{ref},P}(t_2^P)$, with $\vartheta^{\mathrm{ref},P}(t_1)=\vartheta(t_1)$ and $\dot{\vartheta}^{\mathrm{ref},P}(t_1)^\flat\vert_{T\partial M}=\dot{\vartheta}(t_1)^\flat\vert_{T\partial M}$;
\item $\vartheta^{\mathrm{ref},S}:[t_1,t_2^S]$ is a null-geodesic in $(\overline{M},\overline{g}_S)$ connecting two boundary points $\vartheta^{\mathrm{ref},S}(t_1)$ and $\vartheta^{\mathrm{ref},S}(t_2^S)$, with $\vartheta^{\mathrm{ref},S}(t_1)=\vartheta(t_1)$ and $\dot{\vartheta}^{\mathrm{ref},S}(t_1)^\flat\vert_{T\partial M}=\dot{\vartheta}(t_1)^\flat\vert_{T\partial M}$.
\end{itemize}

Following the above procedure, we extend $\vartheta^{(1)}$ and $\vartheta^{(2)}$ to $\Upsilon^{(1)}$ and $\Upsilon^{(2)}$, and extend $\vartheta^{(0)}$ backward to $\Upsilon^{(0)}$. Notice that
\[
\Upsilon^{(1)}([t_-^{(1)},\frac{T}{2}])=\vartheta^{(1)}([t_-^{(1)},\frac{T}{2}]),\quad \Upsilon^{(2)}([t_-^{(2)},\frac{T}{2}])=\vartheta^{(2)}([t_-^{(2)},\frac{T}{2}]),\quad \Upsilon^{(0)}([\frac{T}{2},t_-^{(0)}])=\vartheta^{(0)}([\frac{T}{2},t_-^{(0)}]).
\]
 Since $\gamma^{(1)}([t_-^{(1)},\frac{T}{2}])$ has no conjugate points, $\vartheta^{(1)}([t_-^{(1)},\frac{T}{2}])$ and $\vartheta^{(2)}([t_-^{(2)},\frac{T}{2}])$ intersect only at $p$. Because $\gamma^{(0)}:[\frac{T}{2},t_-^{(0)}]\rightarrow\Omega$ is length-minimizing (w.r.t. the metric $g_P$) for any two points on it and $c_P>c_S$, so $\Upsilon^{(j)}$ $(j=1,2)$ can not intersect $\vartheta^{(0)}((\frac{T}{2},t_-^{(0)}])$. In conclusion, we know that $\Upsilon^{(1)},\Upsilon^{(2)}, \Upsilon^{(0)}$ intersect only at the point $q$, that is,
 \[
 \Upsilon^{(1)}\cap \Upsilon^{(2)}\cap \Upsilon^{(0)}=\{q\}.
 \]

Now, construct $u^{(j)},\widetilde{u}^{(j)}$, $j=1,2,0$ as in Section \ref{threegaussians}. Consider $u^{(j)}$ first, we can write
\[
u^{(j)}=u^{(j)}_{\kappa_j\varrho}+R^{(j)}_{\kappa_j\varrho},\quad u^{(j)}_{\kappa_j\varrho}=u^{(j),\mathrm{inc}}_{\kappa_j\varrho}+u^{(j),\mathrm{ref}}_{\kappa_j\varrho},
\]
for $j=1,2,0$, where $u^{(j),\mathrm{inc}}_{\kappa_j\varrho}$ is a Gaussian beam solution concentrating near $\vartheta^{(j)}$. We have similar expressions for $\widetilde{u}^{(j)}$.
We note that $u^{(1),\mathrm{inc}}_{\kappa_j\varrho},u^{(2),\mathrm{inc}}_{\kappa_j\varrho}$ represent two incident \textit{S}-waves, $u^{(0),\mathrm{inc}}_{\kappa_j\varrho}$ represents an incident \textit{P}-wave. By the above consideration, the solutions can be constructed such that the intersection of the supports of $u^{(1)}_{\kappa_j\varrho},u^{(2)}_{\kappa_j\varrho},u^{(0)}_{\kappa_j\varrho}$ is a small neighborhood of $p$. By the constructions in Section \ref{gaussianbeams}, we have
\[
\begin{split}
\mathbf{a}^{(1)}\vert_{\vartheta^{(1)}}=&\det(Y_S^{(1)})^{-1/2}c_S^{-1/2}\rho^{-1/2}\mathbf{e}^{(1)},\\
\mathbf{a}^{(2)}\vert_{\vartheta^{(2)}}=&\det(Y_S^{(2)})^{-1/2}c_S^{-1/2}\rho^{-1/2}\mathbf{e}^{(2)},\\
\mathbf{a}^{(0)}\vert_{\vartheta^{(0)}}=&\det(Y_P^{(0)})^{-1/2}c_P^{-1/2}\rho^{-1/2}\nabla\varphi^{(0)},\
\end{split}
\]
where $\mathbf{e}^{(j)}$ is a parallel vector field on $\gamma^{(j)}$, which is normal to $\gamma^{(j)}$, for $j=1,2$.
We can, without loss of generality, assume that
\[
\nabla\varphi^{(j)}(q)=c_S^{-1}(q)\xi^{(j)}, \quad \mathbf{e}^{(j)}(q)=c_S^{-1}(q)\alpha^{(j)},\quad j=1,2,\quad \nabla\varphi^{(0)}(q)=c_P^{-1}(q)\xi^{(0)},
\]
with $\alpha^{(j)}\in\mathbb{R}^3$, $|\alpha^{(j)}|=1$, and $\alpha^{(j)}\perp\xi^{(j)}$.

With the above choice of $\zeta^{(1)},\zeta^{(2)},\zeta^{(0)}$ we have (cf. \cite[Lemma 4]{uhlmann2021inverse})
\begin{enumerate}
\item $S(q)=0$;
\item $(\partial_tS(q),\nabla S(q)) =0$;
\item $\Im S(x)\geq cd(x,q)^2$ for $x$ in a neighborhood of $q$, where $c>0$ is a constant,
\end{enumerate}
where $S$ is defined in \eqref{def_S}.
Substituting these solutions into \eqref{eqaulityI}, and using the method of stationary phase, we end up with (cf. \eqref{Iasym})
\[
\varrho^{-1}\frac{\mathrm{i}}{\kappa_1\kappa_2\kappa_3}\mathcal{I}=c_0\mathcal{A}(q)+\mathcal{O}(\varrho^{-1/2})=c_0\widetilde{\mathcal{A}}(q)+\mathcal{O}(\varrho^{-1/2})=\varrho^{-1}\frac{\mathrm{i}}{\kappa_1\kappa_2\kappa_3}\widetilde{\mathcal{I}},
\]
with some constant $c_0\neq 0$.
Letting $\varrho\rightarrow+\infty$, we have
\begin{equation}\label{id_A}
\mathcal{A}(q)=\widetilde{\mathcal{A}}(q)
\end{equation}
We refer to \cite{uhlmann2021inverse} for more details.
\subsection{Determination of $\rho^{-3/2}(\lambda+\mathscr{B})$ and $\rho^{-3/2}(4\mu+\mathscr{A})$}
We assume $\alpha^{(1)}=\alpha^{(2)}=\alpha$ and $\alpha\perp\mathrm{span}\{\xi^{(2)},\xi^{(0)}\}$. Since $\xi^{(1)}\in\mathrm{span}\{\xi^{(2)},\xi^{(0)}\}$, $\alpha\cdot\xi^{(1)}=0$.
Now we compute

\[
\begin{split}
&\det(Y_S^{(1)})^{1/2}\det(Y_S^{(2)})^{1/2}\det(Y_P^{(0)})^{1/2}c_P^{5/2}c_S^{5}\rho^{3/2}\mathcal{A}(q)\\
=&\mathscr{B}(x_0)[(\alpha\cdot\xi^{(1)})(\alpha\cdot\xi^{(0)})(\xi^{(0)}\cdot\xi^{(2)})+(\alpha\cdot\xi^{(2)})(\alpha\cdot\xi^{(0)})(\xi^{(0)}\cdot\xi^{(1)})\\
&\quad\quad\quad\quad+ (\alpha\cdot\xi^{(1)})(\alpha\cdot\xi^{(2)})(\xi^{(0)}\cdot\xi^{(0)})]\\
&+\frac{\mathscr{A}}{4}(x_0)\left((\alpha\cdot\xi^{(1)})(\alpha^{(1)}\cdot\xi^{(0)})(\xi^{(0)}\cdot\xi^{(2)})+(\alpha\cdot\xi^{(2)})(\alpha\cdot\xi^{(0)})(\xi^{(0)}\cdot\xi^{(1)})\right)\\
&+(\lambda+\mathscr{B})(x_0)[(\alpha\cdot\xi^{(1)})(\alpha\cdot\xi^{(0)})(\xi^{(2)}\cdot\xi^{(0)})+(\alpha\cdot\xi^{(2)})(\alpha\cdot\xi^{(0)})(\xi^{(1)}\cdot\xi^{(0)})\\
&\quad\quad+(\alpha\cdot\alpha)(\xi^{(1)}\cdot\xi^{(2)})(\xi^{(0)}\cdot\xi^{(0)})]+2\mathscr{C}(x_0)(\alpha\cdot\xi^{(1)})(\alpha\cdot\xi^{(2)})(\xi^{(0)}\cdot\xi^{(0)})\\
&+(\mu+\frac{\mathscr{A}}{4})(x_0)\Big((\alpha\cdot\alpha)(\xi^{(1)}\cdot\xi^{(0)})(\xi^{(2)}\cdot\xi^{(0)})+(\xi^{(1)}\cdot\xi^{(2)})(\alpha\cdot\xi^{(0)})(\alpha\cdot\xi^{(0)})\\
&\quad \quad+(\alpha\cdot\alpha)(\xi^{(2)}\cdot\xi^{(0)})(\xi^{(1)}\cdot\xi^{(0)})+(\alpha\cdot\xi^{(0)})(\alpha\cdot\xi^{(0)})(\xi^{(1)}\cdot\xi^{(2)})\\
&\quad\quad+(\xi^{(1)}\cdot\alpha)(\xi^{(2)}\cdot\xi^{(0)})(\alpha\cdot\xi^{(0)})+(\xi^{(2)}\cdot\alpha)(\xi^{(1)}\cdot\xi^{(0)})(\alpha\cdot\xi^{(0)})\Big)\\
=& (\lambda+\mathscr{B})(x_0)\xi^{(1)}\cdot\xi^{(2)}+(2\mu+\frac{\mathscr{A}}{2})(x_0)(\xi^{(1)}\cdot \xi^{(0)})(\xi^{(2)}\cdot \xi^{(0)}).
\end{split}
\]
Assume that
\[
\xi^{(1)}\cdot\xi^{(0)}=\cos\psi,
\]
then
\[
\begin{split}
 &\xi^{(1)}\cdot \xi^{(2)}=\xi^{(1)}\cdot(a\xi^{(1)}+\xi^{(0)})=a|\xi^{(1)}|^2+\xi^{(1)}\cdot\xi^{(0)}=\frac{c_P^2-c_S^2}{2(c_S^2\cos\psi+c_Sc_P)}+\cos\psi,\quad \\
&\xi^{(2)}\cdot\xi^{(0)}=(a\xi^{(1)}+\xi^{(0)})\cdot \xi^{(0)} =a\xi^{(1)}\cdot\xi^{(0)}+|\xi^{(0)}|^2=1+\frac{c_P^2-c_S^2}{2(c_S^2\cos\psi+c_Sc_P)}\cos\psi.
\end{split}
\]
Thus we have
\[
\begin{split}
&\det(Y_S^{(1)})^{1/2}\det(Y_S^{(2)})^{1/2}\det(Y_P^{(0)})^{1/2}c_P^{5/2}c_S^{5}\rho^{3/2}\mathcal{A}(q)\\
=&(\lambda+\mathscr{B})\left[\frac{c_P^2-c_S^2}{2(c_S^2\cos\psi+c_Sc_P)}+\cos\psi\right]+(2\mu+\frac{\mathscr{A}}{2})\left[1+\frac{c_P^2-c_S^2}{2(c_S^2\cos\psi+c_Sc_P)}\cos\psi\right]\cos\psi.
\end{split}
\]
Similarly, if we replace $u^{(j)}$ by $\widetilde{u}^{(j)}$ in the above procedure, and notice that $u^{(j),\mathrm{inc}}$ and $\widetilde{u}^{(j),\mathrm{inc}}$ are Gaussian beam solutions along the same null-geodesic $\vartheta^{(j)}$, we have
\[
\begin{split}
\widetilde{\mathbf{a}}^{(1)}\vert_{\vartheta^{(1)}}=&\det(Y_S^{(1)})^{-1/2}c_S^{-1/2}\widetilde{\rho}^{-1/2}\mathbf{e}^{(1)},\\
\widetilde{\mathbf{a}}^{(2)}\vert_{\vartheta^{(2)}}=&\det(Y_S^{(2)})^{-1/2}c_S^{-1/2}\widetilde{\rho}^{-1/2}\mathbf{e}^{(2)},\\
\widetilde{\mathbf{a}}^{(0)}\vert_{\vartheta^{(0)}}=&\det(Y_P^{(0)})^{-1/2}c_P^{-1/2}\widetilde{\rho}^{-1/2}\nabla\varphi^{(0)},\
\end{split}
\]
Then have
\[
\begin{split}
&\det(Y_S^{(1)})^{1/2}\det(Y_S^{(2)})^{1/2}\det(Y_P^{(0)})^{1/2}c_P^{5/2}c_S^{5}\widetilde{\rho}^{3/2}\widetilde{\mathcal{A}}(q)\\
=&\left((\widetilde{\lambda}+\widetilde{\mathscr{B}})\left[\frac{c_P^2-c_S^2}{2(c_S^2\cos\psi+c_Sc_P)}+\cos\psi\right]+(2\widetilde{\mu}+\frac{\widetilde{\mathscr{A}}}{2})\left[1+\frac{c_P^2-c_S^2}{2(c_S^2\cos\psi+c_Sc_P)}\cos\psi\right]\cos\psi\right)(x_0).
\end{split}
\]
Consequently, we have the following identity from \eqref{id_A},
\[
\begin{split}
&\left((\lambda+\mathscr{B})\left[\frac{c_P^2-c_S^2}{2(c_S^2\cos\psi+c_Sc_P)}+\cos\psi\right]+(2\mu+\frac{\mathscr{A}}{2})\left[1+\frac{c_P^2-c_S^2}{2(c_S^2\cos\psi+c_Sc_P)}\cos\psi\right]\cos\psi\right)\rho^{-3/2}(x_0)\\
=&\left((\widetilde{\lambda}+\widetilde{\mathscr{B}})\left[\frac{c_P^2-c_S^2}{2(c_S^2\cos\psi+c_Sc_P)}+\cos\psi\right]+(2\widetilde{\mu}+\frac{\widetilde{\mathscr{A}}}{2})\left[1+\frac{c_P^2-c_S^2}{2(c_S^2\cos\psi+c_Sc_P)}\cos\psi\right]\cos\psi\right)\widetilde{\rho}^{-3/2}(x_0)
\end{split}
\]
By varying $\psi$ in the above identity, we end up with
\begin{equation}\label{equality3}
\rho^{-3/2}(\lambda+\mathscr{B})(x_0)=\widetilde{\rho}^{-3/2}(\widetilde{\lambda}+\widetilde{\mathscr{B}})(x_0),\quad\quad
 \rho^{-3/2}(4\mu+\mathscr{A})(x_0)= \widetilde{\rho}^{-3/2}(4\widetilde{\mu}+\widetilde{\mathscr{A}})(x_0).
 \end{equation}
 Since $x_0$ is an arbitrary point in $\Omega$, and have
 \[
 \rho^{-3/2}(4\mu+\mathscr{A})= \widetilde{\rho}^{-3/2}(4\widetilde{\mu}+\widetilde{\mathscr{A}}),\quad\text{in }\overline{\Omega}.
 \]
  \subsection{Determination of $\left(2\mu+2\mathscr{B}+\mathscr{A}\right)\rho^{-3/2}$}
 We can rescale $\xi^{(2)}$ such that $|\xi^{(2)}|=1$.
Assume  $\alpha^{(1)},\alpha^{(2)}\in\mathrm{span}\{\xi^{(2)},\xi^{(0)}\}$, $|\alpha^{(1)}|=|\alpha^{(2)}|=1$ and $\alpha^{(1)}\perp\xi^{(1)}$, $\alpha^{(2)}\perp\xi^{(2)}$. This means that $\xi^{(1)},\xi^{(2)},\xi^{(0)},\alpha^{(1)},\alpha^{(2)}$ are all in the same plane. Still denote
\[
\xi^{(1)}\cdot\xi^{(0)}=\cos\psi,\quad \xi^{(1)}\cdot\xi^{(2)}=\cos\alpha.
\]
Then we have
\[
\begin{split}
\alpha^{(2)}\cdot\xi^{(1)}&=\sin\alpha,\quad\alpha^{(1)}\cdot \xi^{(2)}=-\sin\alpha,\quad\alpha^{(1)}\cdot\xi^{(0)}=-\sin\psi,\\
&\alpha^{(2)}\cdot\xi^{(0)}=\sin(\alpha-\psi),\quad\alpha^{(1)}\cdot\alpha^{(2)}=\cos\alpha.
\end{split}
\]
Now we calculate
\[
\begin{split}
&\det(Y_S^{(1)})^{1/2}\det(Y_S^{(2)})^{1/2}\det(Y_P^{(0)})^{1/2}c_P^{5/2}c_S^{5}\rho^{3/2}\mathcal{A}(p)\\
=&\mathscr{B}(x_0)(\alpha^{(2)}\cdot\xi^{(1)})(\alpha^{(1)}\cdot\xi^{(2)})(\xi^{(0)}\cdot\xi^{(0)})\\
&+\frac{\mathscr{A}}{4}(x_0)\left((\alpha^{(2)}\cdot\xi^{(1)})(\alpha^{(1)}\cdot\xi^{(0)})(\xi^{(0)}\cdot\xi^{(2)})+(\alpha^{(1)}\cdot\xi^{(2)})(\alpha^{(2)}\cdot\xi^{(0)})(\xi^{(0)}\cdot\xi^{(1)})\right)\\
&+(\lambda+\mathscr{B})(x_0)(\alpha^{(1)}\cdot\alpha^{(2)})(\xi^{(1)}\cdot\xi^{(2)})(\xi^{(0)}\cdot\xi^{(0)})\\
&+(\mu+\frac{\mathscr{A}}{4})(x_0)\Big((\alpha^{(1)}\cdot\alpha^{(2)})(\xi^{(1)}\cdot\xi^{(0)})(\xi^{(2)}\cdot\xi^{(0)})+(\xi^{(1)}\cdot\xi^{(2)})(\alpha^{(1)}\cdot\xi^{(0)})(\alpha^{(2)}\cdot\xi^{(0)})\\
&\quad \quad+(\alpha^{(1)}\cdot\alpha^{(2)})(\xi^{(2)}\cdot\xi^{(0)})(\xi^{(1)}\cdot\xi^{(0)})+(\alpha^{(2)}\cdot\xi^{(0)})(\alpha^{(1)}\cdot\xi^{(0)})(\xi^{(1)}\cdot\xi^{(2)})\\
&\quad\quad+(\xi^{(1)}\cdot\alpha^{(2)})(\xi^{(2)}\cdot\xi^{(0)})(\alpha^{(1)}\cdot\xi^{(0)})+(\xi^{(2)}\cdot\alpha^{(1)})(\xi^{(1)}\cdot\xi^{(0)})(\alpha^{(2)}\cdot\xi^{(0)})\Big)\\
=&-\mathscr{B}(x_0)\sin^2\alpha+\frac{\mathscr{A}}{4}(x_0)\left(-\sin\alpha\sin\psi\cos(\alpha-\psi)-\sin\alpha\sin(\alpha-\psi)\cos\psi\right)\\
&+(\lambda+\mathscr{B})(x_0)\cos^2\alpha\\
&+(\mu+\frac{\mathscr{A}}{4})(x_0)\Big(2\cos\alpha\cos\psi\cos(\alpha-\psi)-2\cos\alpha\sin\psi\sin(\alpha-\psi)\\
&\quad\quad\quad\quad\quad\quad\quad\quad-\sin\alpha\cos(\alpha-\psi)\sin\psi-\sin\alpha\cos\psi\sin(\alpha-\psi)\Big)\\
=&-\mathscr{B}(x_0)\sin^2\alpha-\frac{\mathscr{A}}{4}(x_0)\sin^2\alpha+(\lambda+\mathscr{B})(x_0)\cos^2\alpha+(\mu+\frac{\mathscr{A}}{4})(x_0)\left(2\cos^2\alpha-\sin^2\alpha\right)\\
=&\left(\lambda+2\mu+\mathscr{B}+\frac{\mathscr{A}}{2}\right)(x_0)\cos^2\alpha-\left(\mu+\mathscr{B}+\frac{\mathscr{A}}{2}\right)(x_0)\sin^2\alpha.
\end{split}
\]
Using similar argument as above, we obtain
\begin{equation}\label{equality6}
\begin{split}
&\left(\left(\lambda+2\mu+\mathscr{B}+\frac{\mathscr{A}}{2}\right)\cos^2\alpha-\left(\mu+\mathscr{B}+\frac{\mathscr{A}}{2}\right)\sin^2\alpha\right)\rho^{-3/2}(x_0)\\
=&\left(\left(\widetilde{\lambda}+2\widetilde{\mu}+\widetilde{\mathscr{B}}+\frac{\widetilde{\mathscr{A}}}{2}\right)\cos^2\alpha-\left(\widetilde{\mu}+\widetilde{\mathscr{B}}+\frac{\widetilde{\mathscr{A}}}{2}\right)\sin^2\alpha\right)\widetilde{\rho}^{-3/2}(x_0).
\end{split}
\end{equation}
By \eqref{equality3}, we already have
\begin{equation}\label{equality4}
\left(\lambda+2\mu+\mathscr{B}+\frac{\mathscr{A}}{2}\right)\rho^{-3/2}(x_0)=\left(\widetilde{\lambda}+2\widetilde{\mu}+\widetilde{\mathscr{B}}+\frac{\widetilde{\mathscr{A}}}{2}\right)\widetilde{\rho}^{-3/2}(x_0).
\end{equation}
For $\alpha\neq 0$, we obtain from \eqref{equality6} and \eqref{equality4}
\begin{equation}\label{equality5}
\left(2\mu+2\mathscr{B}+\mathscr{A}\right)\rho^{-3/2}(x_0)=\left(2\widetilde{\mu}+2\widetilde{\mathscr{B}}+\widetilde{\mathscr{A}}\right)\widetilde{\rho}^{-3/2}(x_0).
\end{equation}

\subsection{Determination of $\mathscr{C}$}
Using \eqref{equality4} and \eqref{equality5}, we can conclude that
\[
\rho^{-3/2}(\lambda+\mu)=\widetilde{\rho}^{-3/2}(\widetilde{\lambda}+\widetilde{\mu})\quad \text{in }\overline{\Omega}.
\]
 Since $\frac{\mu}{\rho}=\frac{\widetilde{\mu}}{\widetilde{\rho}}$ and $\frac{\lambda}{\rho}=\frac{\widetilde{\lambda}}{\widetilde{\rho}}$, we have 
 \[
 (\lambda,\mu,\rho)= (\widetilde{\lambda},\widetilde{\mu},\widetilde{\rho})\quad \text{in }\overline{\Omega}.
 \] 
Using \eqref{equality3}, we have
 \[
 (\mathscr{A},\mathscr{B})= (\widetilde{\mathscr{A}},\widetilde{\mathscr{B}})\quad \text{in }\overline{\Omega}.
 \]
 
For the final step, we can apply the result in \cite{uhlmann2021inverse} to obtain
\[
\mathscr{C}=\widetilde{\mathscr{C}}\quad \text{in }\overline{\Omega}
\]
by inverting a \textit{Jacobi-weighted ray transform of the first kind} on $(\Omega,g_P)$. See also \cite{acosta2022nonlinear} for more details.

Invertibility of this kind of weighted ray transform is proved in \cite{feizmohammadi2019inverse} under the no conjugate points assumption. 

Now assume that $(\Omega,g_P)$ satisfies the the foliation condition. Recall that $\partial\Omega$ is strictly convex with respect to $g_P$. Let $\varpi\in C^\infty(\widetilde{\Omega})$ be a defining function of $\partial\Omega$, such that $\varpi>0$ in $\Omega\setminus\partial\Omega$, $\varphi<0$ in $\widetilde{\Omega}\setminus\overline{\Omega}$, and $\varpi$ vanishes non-degenerately on $\partial\Omega$. For any $p\in\partial\Omega$, there exists a function $\tilde{x}\in C^\infty(\widetilde{\Omega})$ with $\tilde{x}(p)=0$, $\mathrm{d}\tilde{x}(p)=-\mathrm{d}\varpi(p)$, such that for $c>0$ small enough, $O_p=\{\tilde{x}>-c\}$ has no conjugate points. Then the above Jacobi-weighted ray transform is invertible on $O_p$.  A layer stripping scheme can be used to derive global invertibility \cite{uhlmann2016inverse,paternain2019geodesic}.
\bibliographystyle{abbrv}
\bibliography{biblio}

\end{document}